\documentclass[twoside, 12pt]{amsart}
\usepackage{latexsym}
\usepackage{amssymb,amsmath,amsopn}
\usepackage[dvips]{graphicx}   
\usepackage{color,epsfig}      
\usepackage{url}

\makeatletter
\def\@settitle{\begin{center}%
  \baselineskip14\p@\relax
  \bfseries
  \uppercasenonmath\@title
  \@title
  \ifx\@subtitle\@empty\else
     \\[1ex]\uppercasenonmath\@subtitle
     \footnotesize\mdseries\@subtitle
  \fi
  \end{center}%
}
\def\subtitle#1{\gdef\@subtitle{#1}}
\def\@subtitle{}
\makeatother

\newtheorem{theorem}{Theorem}
\newtheorem{proposition}{Proposition}[section]

\newtheorem{statement}{Statement}
\newtheorem{defi}{Definition}

\newtheorem{cor}{Corollary}

\DeclareMathOperator{\aff}{aff}

\begin{document}
\title[]{Strongly self-dual polytopes}
\subtitle{}
\author[\'A. G.Horv\'ath]{\'Akos G.Horv\'ath}
\address {\'A. G.Horv\'ath \\ Department of Algebra and Geometry \\ Mathematical Institute \\
Budapest University of Technology and Economics\\
H-1521 Budapest\\
Hungary}
\email{ghorvath@math.bme.hu}

\author[I. Prok]{Istv\'an Prok}
\address {Department of Algebra and Geometry \\Mathematical Institute \\
Budapest University of Technology and Economics\\
H-1521 Budapest\\
Hungary}
\email{prok@math.bme.hu}


\begin{abstract}
This article aims to study the class of strongly self-dual polytopes (ssd-polytopes for short), defined in a paper by Lovász \cite{lovasz}. He described a series of such polytopes (called $L$-type polytopes), which he used to solve a combinatorial problem. From a geometrical point of view, there are interesting questions: what additional elements of this class exist, and are there any with a different structure from the $L$-type ones? We show that in dimension three, one of their faces defines $L$-type polyhedra. Illustrating the algorithm of the proof, we present an ssd-polytope of 23 vertices whose combinatorial structure differ from those of $L$-type ones. Finally, with an elementary discussion, we prove that for fewer than nine vertices, there are only fifth one ssd-polyhedra, four of them can be constructed by Lovász's method, and we can find the fifth one with "the diameter gradient flow algorithm" of Katz, Memoli and Wang \cite{katz-memoli-wang}. 
\end{abstract}

\maketitle

\section{Introduction}

Lov\'asz in \cite{lovasz} defines the notion of a strongly self-dual polyhedron. He used this class of polytopes to investigate the chromatic number of the graph $G(n,\alpha)$ defined on the $n$-dimensional unit sphere. He proved the following theorem:

\begin{theorem}[\cite{lovasz}]\label{thm:theorem2}
The chromatic number of the graph formed by the principal diagonals of a strongly self-dual polytope in $\mathbb{R}^n$ is greater than or equal to $n+1$. Equality holds if its parameter is greater than the diameter of the monochromatic parts of the colouring defined by the regular simplex of dimension $n$.
\end{theorem}

 In this article, we present ssd-polytopes whose combinatorics differ from the known ones. Our goal is also to classify those three-dimensional ssd-polyhedra that have few vertices. Lov\'asz mentioned that strongly self-dual polygons are exactly the odd regular polygons. The ssd-polyhedra are closely related to the self-polar convex polytopes; they are negatively (or anti-) self-polar polytopes that can write in the unit sphere. Despite the existing self-similarity theory, there is no three-dimensional compact classification result in this direction. We can find some nice  results on ssd-polytopes in the paper \cite{katz} by Katz. He introduced the concept of point-wise extremal sets of the unit sphere and proved that the convex hull of such a set is an ssd-polytope in the case when its spherical diameter is less than $\frac{2\pi}{3}$. Since we are interested in all ssd-polyhedra (without constraint on their diameter), this paper only partially helps. A recent paper in this direction by Jensen \cite{jensen} is also motivated by Lov\'asz's work. It contains interesting information in self-polar and also negatively self-polar polytopes. The terminology of this paper different from the current one, defining ssd-polytopes as negatively self-polar polytopes whose vertices are equidistant from the origin. The constraint that the vertices belong to a given sphere is powerful, so Jensen's article does not contain much information on the characterization of ssd-polytopes. In dimension three, Jensen defined three negative self-polar polytopes, two of which can be obtained as $L$-type ssd-polytopes, too. Still, the elements of the third series are such negatively self-polar polyhedra that they cannot be represented as ssd-polytopes. (Since coordinates give the vertices of these polyhedra, it is easy to check that there is no circumscribed ball of the elements of this sequence.) By this construction, Jensen proved that if $n=4$ or $n\geq 6$, there is a negatively self-polar polyhedron with $n$ vertices. Jensen formulated some questions, from which \emph{Question 9.5.} an existence problem: \emph{For which values of $r$ and $d$ does there exist a negatively self-polar polytope in $R^d$ such that its vertices are $r$ away from the origin?}
 
The most recent paper in this direction is \cite{katz-memoli-wang} written by Katz, Memoli and Wang. By a computer-aided algorithm, the authors find 64 pointwise extremal configurations of the unit sphere, which realize ssd-polytopes up to 10 vertices. They implemented the diameter gradient flow algorithm on the sphere from given configurations. They found a list of ssd-polytopes up to 10 vertices. 

In Theorem \ref{thm:uptoeight}, we prove that the above list is complete for at most eight vertices. We also justify in Theorem \ref{thm:determination} that one of their faces determines the $L$-type polyhedra. Using the algorithm of the proof of Theorem \ref{thm:determination}, we construct a new 23-vertex ssd-polyhedron, which neither $L$-type ssd-polyhedron nor holds the constraint on its diameter, required in \cite{katz}. This fact demonstrates that there are ssd-polytopes in which the vertex set is not a pointwise extremal set with a diameter less or equal to $2\pi/3$.

\section{Strongly self-dual three-dimensional polyhedra}

L. Lov\'asz in \cite{lovasz} introduced the class of strongly self-dual polytopes. (Another used terminology is that it is a negatively self-polar polyhedron inscribed in the unit ball.)
\begin{defi}
Let $P$ be a convex polytope in $\mathbb{R}^n$. We say that $P$ is \emph{strongly self-dual} if the following conditions hold.
\begin{itemize}
\item $P$ is inscribed in the $n-1$-dimensional unit sphere $S^{n-1}$ (so that all vertices of $P$ lie in the sphere $S^{n-1}$);
\item $P$ is circumscribed around a sphere $S'$ with center $O$ and radius $r$, where $0<r<1$ (so that $S'$ touches every facets of $P$);
\item There is a bijection $\sigma$ between vertices and facets such that if $v$ is any vertex then the facet $\sigma(v)$ is orthogonal to $v$.
\end{itemize}
\end{defi}

The following statement can reformulate the definition:
\begin{statement}[\cite{ghorvath},\cite{jensen}]\label{st:selfpol}
If a polytope $P$ is strongly self-dual, it is a homothetic copy of a self-polar polytope inscribed in a sphere.
\end{statement}

The following theorem collects some general information on strongly self-dual polytopes.

\begin{theorem}[\cite{lovasz}, \cite{ghorvath}]\label{thm:stronglysdpol}
Let $P$ be a strongly self-dual polytope of the $n$-dimensional Euclidean space. Denote by $V(P),|V|$, $E(P), |E|$ and $F(P),|F|$ the set and its cardinality of the vertices, the edges, and the facets, respectively. Then we have
\begin{enumerate}
\item[(a)] If $v_1,v_2\in V(P)$ and $v_1$ is a vertex of $\sigma(v_2)$ then $v_2$ is a vertex of $\sigma(v_1)$. Moreover $v \not \in \sigma (v)$ (see Lemma 4.3 in \cite{jensen}).
\item[(b)] Every principal diagonal of a strongly self-dual polytope is the same length. (Principal diagonals connect $v$ to the vertices of $\sigma(v)$.) Furthermore, the common length $\alpha=\sqrt{2+2r}$ is the parameter of the polytope.
\item[(c)] Let $P$ be a strongly self-dual polytope with parameter $\alpha$. \newline Then $\alpha\geq \sqrt{2(n+1)/n}$.
\item[(d)] For each $n\geq 2$ and $\alpha_1<2$ there exists a strongly self-dual polytope in $\mathbb{R}^n$ with parameter at least $\alpha_1$.
\item[(e)] The duality $\sigma$ puts the $k$-dimensional and $(n-k-1)$-dimensional faces to orthogonal pairs $\{G,\sigma(G)\}$ with a common normal transversal. Moreover if $d_G$ means the distance of the affine hull $\mathrm{aff}(G)$ of the face $G$ from the origin we have
$$
d_Gd_{\sigma(G)}=r=\frac{\alpha^2}{2}-1.
$$
(This result can be found as Lemma 4.5 in \cite{jensen}, too.)
\end{enumerate}
\end{theorem}

We notice that the examined mapping $\sigma$ is precisely the well-known idempotent inclusion reversing bijection of the face lattice of the polytope into the face lattice of the dual polytope (see in  \cite{ziegler}). Hence, if $vw$ any edge of a three-dimensional polyhedron, then $v\subset vw\subset G$ implies the relation $\sigma(v)\supset \sigma(vw)\supset \sigma(G)$. Later, we use this fact without further comments.

The equality in (e) says that the product of the distances of the opposite faces from the origin is independent of the choice of the face. Furthermore, it can be expressed by the parameter of $P$.
We now prove that property (e) can be extended for any diagonal of an ssd-polytope, precisely we have:

\begin{proposition}\label{prop:diagonal}
If $(uv)$ is a diagonal (or especially an edge) of the ssd-polytope $P$, then the following equality holds:
$$
d_{(uv)}d_{\sigma(u)\cap\sigma(v)}=r.
$$ 
\end{proposition}

\begin{proof} 
Consider the intersection of the $2$-dimensional plane $Ouv$ with the $(n-1)$-dimensional hyperplanes $\sigma(u)$ and $\sigma(v)$ (see Fig. \ref{fig:diagonal}).

\begin{figure}
    \centering
    \includegraphics{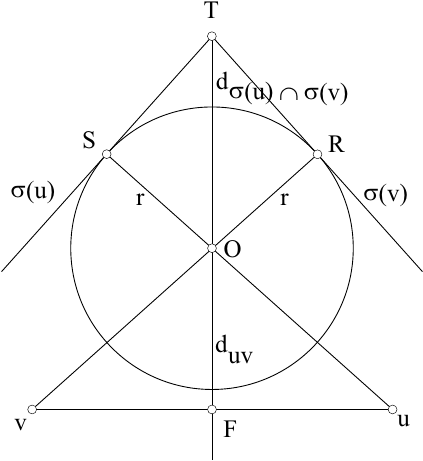}
    \caption{The dual of a diagonal}
    \label{fig:diagonal}
\end{figure}

Let $F$, $R$, $S$, and $T$ be the middle of the diagonal, the tangent points of the dual planes at the sphere of radius $r$ and the intersection point of the plane $Ouv$ and the $(n-2)$-dimensional plane $\sigma(u)\cap\sigma(v)$, respectively. Since $|OT|=d_{\sigma(u)\cap\sigma(v)}$, $|OF|=d_{(uv)}$, from the similar triangles $\triangle_{OuF}$ and $\triangle_{OTS}$ the statement follows.

\end{proof}

\begin{cor}\label{cor:diagonal} In dimension $3$, if $(uv)$ is a diagonal of a face $F$ of the polyhedron then $d_{(uv)}\geq r$, and so $|OT|=d_{\sigma(u)\cap\sigma(v)}\leq 1$. Hence the circumcircle of $\sigma(u)$ (resp. $\sigma(v)$) with center $S$ (resp. $R$) and radius $\sqrt{1-r^2}$ goes through the points $x,y$ of the unit sphere. By duality $\sigma(x)$ contains $u$ (resp. $v$) hence the diagonal $(uv)$ belongs to $\sigma(x)$. Similarly, $\sigma(y)$ also contains the segment $(uv)$. This implies that either $x=y$, consequently $(uv)$, goes through the centre of $F$ or, at most, one of $x$ and $y$ is a polyhedron vertex, and the other cuts by a face. (Of course, a real possibility that both $x$ and $y$ outer points of the closed polyhedron.) Clearly, we also have $d:=|xu|=|xv|=|yu|=|yv|$, and if one of $x$ and $y$ is a vertex then two from the above segments are principal diagonals with length $\sqrt{2+2r}$. (The distance $d$ of a principal diagonal from the center is $\sqrt{\frac{1-r}{2}}$ which is less than $r$ if $r> \frac{1}{2}$. 

Consequently, the following discussion can determine the type of segment $(uv)$, which joins two vertices. 

\begin{itemize}
\item First, we determine $d_{(uv)}$. If $d_{(uv)}<r$, $(uv)$ is a diagonal of the body.

\item If $d_{(uv)}\geq r$ we determine the points $x$ and $y$,  the lengths $d$ and check whether it is equal to $\sqrt{2+2r}$ or not. 
\begin{itemize}
    \item If the equality does not hold, then $x$ and $y$ are not vertices, and so $(uv)$ is a diagonal. 
    \item If the equality holds, we have two possibilities. $x=y$ and $(uv)$ is a diagonal of a face through its center, or $x\not = y$, and $(uv)$ is an edge of the polyhedron.
 \end{itemize}
\end{itemize}
\end{cor}

\subsection{On the construction of Lov\'asz}

\begin{figure}[ht]
\includegraphics[scale=1]{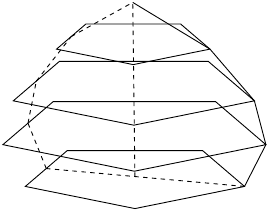}
\caption[]{Parallels and a meridian of $P(4,5)$.}
\label{fig:lovaszpoly}
\end{figure}

We note that in \cite{lovasz}, there is a construction to give a large class of $n$-dimensional strongly self-dual bodies where $n\geq 3$.
\emph{Let $C$ be the unit circle in ${\bf R}^2$ and let $E$ be an ellipse with axes $2$ and $2r_0$. Thus, $E$ touches $C$ in two points $x$ and $y$. Choose any $t$ with $r_0>t>r_1$ and let $C_t$ denote the circle concentric with $C$ and radius $t$. It is clear by a continuity argument that $t$ can be chosen so that we can inscribe an odd polygon $Q=(x_0=x,\ldots, x_{2k+1}=x)$ in $E$ so that the sides of $Q$ are tangent to $C_t$. Let $\alpha$ be an orthogonal affine transformation mapping $E$ on $C$ and let $y_0=x_0,y_1,\ldots, y_{2k+1}=x_0$ be the images of $x_0,x_1,\ldots ,x_{2k+1}$ under $\alpha$. Consider $C$ as the "meridian" of $S^{n-1}$ with $x$ as the "north pole". Let $S^{n-2}$ be the "equator" and suppose that $P_0$ is an $(n-1)$-dimensional ssd-polytope inscribed in the "equator". Let for each vertex $v$ of $P_0$, $M_v$ be the "meridian" through $v$ (so $M_v$ is a one-dimensional semicircle). Let $L_i$ denote the "parallel" through $y_i$ ($i=1,\ldots,k$). 
We denote by $u(v, i)$ the intersection point of $M_v$ and $L_i$. Further, let $u(v,0)=x$ for all $v$. We define the polytope
$$
P=\mathrm{conv}\{{u(v,i)}: v\in V(P_0), i=0,\ldots, k\}.
$$
(Here $V(P_0)$ denotes the set of vertices of $P_0$.)}

In dimension three, its base (namely $P_0$) is an odd regular $l$-polygon, and the meridians through a vertex of the base are $2k+1$-polygons with one common vertex, which is of degree $l$. (Note that only $k+1$ vertices of such a meridian are also the vertices of $P(k,l)$. The remaining $k$ are middle points of certain parallel edges of $P(k,l)$. (See Figure \ref{fig:lovaszpoly}.)

The construction has something like pyramidal building up; it allows one regular face with $l$ vertices, $l$ triangles, and many quadrangles. The quadrangles are symmetric trapezoids connecting an edge of the regular odd $l$-polygon with one triangle at the opposite vertex of the polyhedron. Of course, this latter triangle is the image (by the bijection $\sigma$) of that vertex of the
$l$-gon, which lies opposite the given edge in the polygon. Such a polytope is called by \emph{L-type ssd-polytope}. First, we list some exceptional cases of $L$-type polytopes. Denote by $P(k,3)$ the $L$-type ssd-polyhedron with base the regular triangle and $k$ layers.

\begin{proposition}\label{prop:regulartriangleface}
   Assume that the $L$-type polyhedron has a regular triangular face at its top vertex $x=u(v,0)$. Then, it is a regular simplex. 
\end{proposition}

\begin{proof}
We prove that for $k\geq 2$, the polyhedron $P(k,3)$ contains only one regular triangular face $F_0$, which is the base of the polyhedron. If at the top vertex $x_0$ one of the triangular faces is regular, then its dual is a vertex $x_1$ of $F_0$ with corresponding edges of equal length. Since every two of the three edges are edges of a face of $P(k,3)$ with a congruent circumscribed circle, we have such a symmetric trapezoid face in which three consecutive edges have the same length is equal to the size of a regular triangle inscribed in this circle. This means that the fourth edge has a length of zero, which is contradictory. Hence, the triangles of $P(k,3)$ at its top vertex $x$ are not regular.

We get that, in the sequence of $L$-type ssd-polyhedron, with $l=3$, the only polyhedron is the regular simplex $P(1,3)$, which holds the requirement of the statement.

Assume that $l\geq 7$ and $k\geq 1$. In this case, the statement's condition follows a symmetric trapezoid face with three equal edges with a length equal to the length of the edge of a regular 7-gon inscribed in its circumscribed circle. This trapezoid doesn't contain the centre of the circumscribed circle, which is impossible for the face of an ssd-polytope. This contradiction proves that $P(k,l)$ for $l\geq 7$ and $k\geq 1$ hasn't a regular triangle face.

Finally, assume that $l=5$. Let $r$ be the radius of the inscribed sphere; hence, the length of the edge of the base pentagon is $2\sqrt{1-r^2}\sin\frac{\pi}{5}$. 

Observe that the angles of the regular spherical triangles at the top vertex are $\frac{2\pi}{5}$ hence from the spherical Cosine theorem on angles gives the equality
$$
\cos \frac{2\pi}{5}=-\cos^2\frac{2\pi}{5}+\sin^2\frac{2\pi}{5}\cos c,
$$
where $c$ is the spherical lengths of the arc $u(j,0)u(j,1)$. Hence 
$$
\cos c =\frac{\sqrt{5}}{5}.
$$
On the other hand, the equal length of the edges of the regular triangles by Proposition \ref{prop:diagonal} is 
$$
|u(j,0)u(j,1)|=2\sqrt{1-\left(\frac{r}{\sqrt{1-(1-r^2)\sin^2\frac{\pi}{5}}}\right)^2}=2\sqrt{\frac{1-r^2}{1+(5-2\sqrt{5})r^2}},
$$
using the formula $\cos\frac{\pi}{5}=\frac{\sqrt{5}+1}{4}$.
We get that from the two values of equality.
$$
2\sqrt{\frac{(1-r^2)}{1+(5 -2\sqrt{5})r^2}}=2\sin{\frac{c}{2}}=2\sqrt{\frac{5-\sqrt{5}}{10}},
$$
from which follows that
$$
r=\sqrt{\frac{5+2\sqrt{5}}{15}}.
$$
We prove that $P(k,5)$ polyhedron with this $r$ doesn't exist. To this we consider the half-meridian with possible vertices $u(1,0)=x_0,u(1,1)\ldots u(1,k-1), u(1,k)$ and 
prove that the recursive setting of $u(1,j)$ couldn't close. Observe that the symmetric trapezoid $u(1,k-1)u(1,k)u(2,k)u(2,k-1)$ there are three edges with equal length (these are $u(1,k-1)u(1,k)$,$u(1,k)u(2,k)$ and $u(2,k)u(2,k-1)$. This length of the regular pentagon can be described in the circle of radius $r$. The spherical length $b$ of the arc $u(1,k-1)u(1,k)$ can be determined from its straight length. Since 
$$
|u(1,k-1)u(1,k)|=\frac{\sqrt{5}+1}{2}r=\sqrt{\frac{10+2\sqrt{5}}{15}},
$$
then 
$$
\sin\frac{b}{2}=\frac{1}{2}\sqrt{\frac{10+2\sqrt{5}}{15}}, \quad 
\cos\frac{b}{2}=\sqrt{\frac{25-\sqrt{5}}{30}} \quad \mbox{ and so } \cos b= \frac{10-\sqrt{5}}{15}.
$$
The spherical length of $u(1,k-1)u(1,k)$ is
$$
b= \mathrm{arccos}\frac{10-\sqrt{5}}{15} \approx 35.339614214104^\circ.
$$
Let $a$ be the angle $KOu(1,k)_\angle$, where $K$ is the centre of the small circle containing the points $u(j,k)$. Then 
$$
a=\mathrm{arccos}\, r=\mathrm{arccos }\sqrt{\frac{5+2\sqrt{5}}{15}} \approx 37.37736814065^\circ .
$$
Since $u(1,0)Ou(1,1)_\angle$ is $c$, then we also have that
$$
c=\mathrm{arccos} \frac{\sqrt{5}}{5} \approx 63.434948822922^\circ ,
$$
and thus
$$
a+b+c \approx 136.151931177676^\circ.
$$
From this we can see that $k\ne 2$ ($u(1,1)\ne u(1,k-1)$). Hence the edge $u(1,1)u(1,2)$ exists and it is the dual of the edge $u(3,k-1)u(4,k-1)$. It is a diagonal of the regular pentagon with edge $u(3,k)u(4,k)$ thus its length is
$$
\frac{\sqrt{5}+1}{2}\sqrt{\frac{10+2\sqrt{5}}{15}}=\sqrt{\frac{15+8\sqrt{5}}{15}}.
$$
As in the case of $u(j,0)u(j,1)$, by Proposition \ref{prop:diagonal} we get the length of $u(1,1)u(1,2)$ as follows:
$$
|u(1,1)u(1,2)|=2\sqrt{1-\left(\frac{r}{\sqrt{1-\frac{15+8\sqrt{5}}{60}}}\right)^2}=2\sqrt{1-\left(\frac{r}{\sqrt{\frac{45-8\sqrt{5}}{60}}}\right)^2}.
$$
But since
$$
r=\sqrt{\frac{5+2\sqrt{5}}{15}}>\sqrt{\frac{45-8\sqrt{5}}{60}},
$$
We get a negative number under the square root, which is impossible. This contradiction proves our statement.
\end{proof}

\begin{cor}\label{cor:regulartriangle}
If we know a regular triangular face of a $3$-dimensional $L$-type ssd-polyhedron, then this face uniquely determines the polyhedron (up to congruences).
\end{cor}

\begin{theorem}\label{thm:determination}
Let $P$ be a L-type ssd-polytope of dimension three, and $F_0$ one of its faces with the vertices $x_1,\ldots,x_k$. Then $P$ is uniquely determined up to congruence by $F_0$.
\end{theorem}

\begin{proof}
The position of the points $x_1,\ldots,x_k$ in the unit sphere $\mathcal{S}$ determines the plane $\Pi_0$ of the face $F_0=F$, the centre $O_0$ of its circumscribed circle, the distance $r$ of $\Pi_0$ and the centre $O$ of $\mathcal{S}$. Hence $r$ is the radius of the inscribed ball and for $x_{k+1}=\mathcal{S}\cap \overrightarrow{O_0O}$ (where $\overrightarrow{O_0O}$ is the half-line with origin $O_0$ through $O$) we have $F_{0}=\sigma(x_{k+1})$. Consider the planes $\mathrm{aff}(\sigma(x_{1})), \mathrm{aff}(\sigma(x_{2})), \ldots, \mathrm{aff}(\sigma(x_{k}))$, respectively. These contain the vertex $x_{k+1}$ and intersect the sphere $\mathcal{S}$ in a chain of circles $C^{k+1}_1,\ldots, C^{k+1}_k$.  These circles intersect at point $x_{k+1}$, so they have another intersection point. Denote by the new points $x_{k+i+1}:=C^{k+1}_{i}\cap C^{k+1}_{i+1}$ where $i=1\ldots k$  (of course $C^{k+1}_{k+1}:=C^{k+1}_1$). 
For an $i$, the point $x_{k+i+1}$ may agree with one of the original vertices, say it is $x_j$ where $i\not= j$. (For example, this situation occurs in the case of ssd-pyramids.) This means that the next step of the algorithm cannot be applied to this vertex $x_j$. In the other cases, the constructed points are new vertices of $P$ with the property of duality $\sigma(x_{i}x_{i+1})=x_{k+1}x_{k+i+1}$.  By Corollary \ref{cor:diagonal} we consider the segments $x_{k+i+1}x_{k+i+2}$ for $i=1,\ldots k$ and determine their types, respectively. Since all of them are in a proper face $F_{i+1}$ of $P$, they are edges or face-diagonals. 

In the second step, we consider the respective dual planes $\aff(\sigma x_{k+i+1})$ and the corresponding circles $C^{k+i+1}_{i}$ for $i=1,\ldots, k$. Two consecutive circles $C^{k+i+1}_{i}$ and $C^{k+i+2}_{i+1}$ intersect each other at the vertex $x_i$ and at most one further point for which we have four possibilities:
\begin{itemize}
    \item It is a vertex of $P$ that occurs in the first step of the algorithm.
    \item It is  a new vertex of $P$ and with $x_i$ forms an edge dual to the edge $(x_{k+i+1}x_{k+i+2})$,
  \item It is out of $P$ because $x_{k+i+1}x_{k+i+2}$ is a face-diagonal,
  \item It is equal to $x_i$ because $(x_{k+i+1}x_{k+i+2})$ is a face- diagonal through the centre of the circle $C^{k+1}_i$.
\end{itemize}
If we find a new vertex, we repeat the second step originating from it. If there is no new vertex, we end the process. Since the second intersection point of two circles is unique, we finally get three determined sets, the set $\mathcal{V}=\{x_1,\ldots, x_{k+1},\ldots x_s\}$ of vertices, the set $\mathcal{E}$ of edges, and a set $\mathcal{D}$ of the set of face-diagonals, respectively. If $\mathcal{D}$ is empty, then we can determine the faces of the polyhedron as follows. Consider the set $\mathcal{V}$ and determine the incidence graph of the vertices and edges. Since our algorithm gives pairs of dual edges, we can determine the subset $\mathcal{E}_1$ of $\mathcal{E}$ containing those edges which are dual of an edge joined with the vertex $x_1$. Since $\mathcal{D}$ is empty, these edges should form a convex polygon that is the dual face $F_1=\sigma(x_1)$ of $x_1$. (If the algorithm produces a broken line inscribed in its circle, then it is stopped at a diagonal, which is an element of $\mathcal{D}$.) In the following step, we consider $x_2$ and determine $\sigma(x_2)=: F_2$, and so on... The process is finished in finite steps and uniquely builds up the polyhedron $P$, proving the statement.

Consider the case when $\mathcal{D}$ is not empty. Then, the first element $d_1$ connects the endpoints of a broken line of edges of a face. This face is either a symmetric trapezoid or the basic polygon, a regular polygon with an odd number of vertices. In both cases, the broken line is the union of at least two edges with a common vertex $v$. The valence of $v$ is three or four. It cannot be greater because a face incident with the up vertex of the polyhedron is a triangle, implying that $d_1$ is an edge. 
\begin{itemize}
    \item  Assume that the valence of $v$ is three. Then its dual is a triangle (from the vertex $v$ arises) in which at least two sides have equal lengths. The dual of these sides have equal lengths. If the length of the third one differs from their length, we conclude that these equal edges span the basic regular polygon of the body, and the third edge in $v$ is a common edge of two congruent trapezoids that contain it. Since we know the circumscribed circle of these faces, we can add new vertices of the set $\mathcal{V}$ to continue the process. Let us assume that the three edges of $v$ have equal lengths. If they form at least two distinct angles, then by the congruence of the trapezoids in the same layer, we can determine again which plane contains the regular polygon and continue the process as above. 
    
    Finally, assume that all edges containing $v$ and all angles formed by these edges are equal, respectively. In this case, the dual of $v$ is a regular triangle, and all faces belong to the vertex $u(v,0)$, which holds this property, too. In this case, the polyhedron is also uniquely determined by Corollary \ref{cor:regulartriangle}.

    \item If the valence of $v$ is four, then in this vertex, either four quadrangles meet and two by two are congruent, or two congruent quadrangles meet two congruent triangles. In both cases, we have two possibilities. If the angles of the consecutive edges in this vertex have two distinct values, these angles and lengths determine the corresponding opposite edges in a horizontal layer and form a regular polygon in their circumscribed circle. When all edges have the same lengths, and the four angles at $v$ are congruent, the dual $\sigma(v)$ is a square. Since the length of the parallel edges in the form:
    $2\sqrt{1-r_i^2}\sin\frac{\pi}{l}$ resp. $2\sqrt{1-r_j^2}\sin\frac{\pi}{l}$  for $P(k,l)$, hence $j=i+1$ and they are in a symmetric position with respect to the equator of the unit sphere. The equator (by duality) also contains $v$, which is impossible because $v$ should belong to one of the $k+1$ layers, while the equator lies between the $i$-th and $i+1$-th layers of the vertices. Hence, this case cannot be realized. 
\end{itemize}
Hence, we can add all the vertices of a face to our construction. With the corresponding new edges, we can continue our process and when stopped, the cardinality of $\mathcal{D}$ should be more minor than in the starting moment. Restarting the algorithm a finite number of times, we get that $\mathcal{D}$ is empty, and the process finishes producing the original polyhedron uniquely.
\end{proof}

\begin{cor}\label{cor:congruency}
   If $P$ and $Q$ are $L$-type ssd-polyhedra with a pair of congruent faces, the two polyhedra are congruent.
\end{cor}

\subsection{Some further observations on ssd-polyhedra}

Using Euler relation $2=|F|+|V|-|E|$ for a strongly self-dual polyhedron, we get that the number of edges $2(|V|-1)$ and so
$$
\sum\limits_{l\geq 3}l\alpha_l=2|E|=4(|V|-1)=4\left(\sum\limits_{l\geq 3}\alpha_l-1\right)
$$
implying
\begin{equation}\label{eq:main}
4=\sum\limits_{l\geq 3}(4-l)\alpha_l=\alpha_3-\sum\limits_{l\geq 5}(l-4)\alpha_l,
\end{equation}
where $\alpha_l$ denotes the number of $l$-gonal faces.
\begin{remark}

We list some natural consequences of the Euler theorem.
\begin{itemize}
    \item For a polyhedron with triangular faces, we have $3|F|=2|E|$ giving $|F|=4$, so the only simplicial strongly self-dual polyhedron is the regular tetrahedron.
    \item For $k>3$ there is no such ssd-polyhedron with $n$ vertices whose all faces are $k$-sided polygon. In fact, by duality, the valences of the vertices are $k$. Hence the number of edges is $\frac{k}{2}n$, from which the number of faces is $(\frac{k}{2}-1)n+2$ which can be self-dual if and only if $2-\frac{k}{2}=\frac{2}{n}$, but this equation cannot be solved because the left side is non-positive for $k\geq 4$.
    \item Assume that a strongly self-dual polyhedron has only two faces for which the number of vertices is $p\geq q>4$.  (Let $P=\sigma(v)$, $Q=\sigma(w)$ be the faces with vertices $p$ and $q$, respectively.) Then, the number of triangle faces from (\ref{eq:main}) is $ \alpha_3=p+q-4$. We immediately get that
    $$
    |V|=|F|=p+q+\alpha_4-2 \mbox{ and } |E|=2p+2q+2\alpha_4-2.
    $$
\end{itemize}
\end{remark}
Finally, we have to mention Theorem 5.5 in \cite{jensen} which says that if a pyramid has a negatively self-polar realisation, then so does the base of the pyramid. Lovasz's result for the plane shows that if a polyhedron is pyramidal, its base is a regular odd polygon.

\subsection{Ssd-polyhedra with up to eight vertices.}
In manuscript \cite{katz-memoli-wang}, there is a list of point-wise extremal configurations of the unit sphere of dimension 2. All table elements realise an ssd-polyhedron of the three-space up to 10 vertices. Using an elementary argument, we prove that this list is complete up to 8 vertices, so we have the following theorem:

\begin{theorem}\label{thm:uptoeight}
There are five strongly self-dual polyhedra up to eight vertices. Four ones were produced by Lov\'asz's method, and the fifth by the diameter gradient flow algorithm introduced in \cite{katz-memoli-wang}. (The resulting types are shown in Fig. \ref{fig:ssdpoly}.) 
\end{theorem}

\begin{figure}[ht]
\includegraphics[scale=0.8]{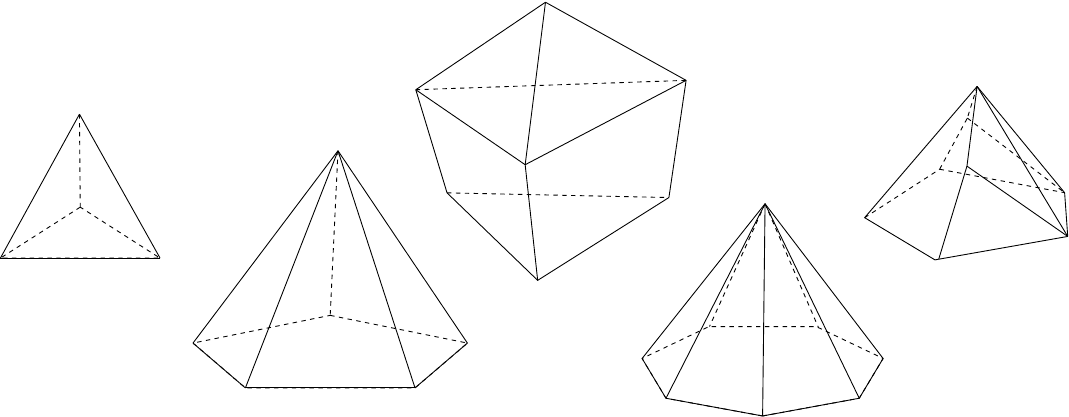}
\caption[]{The strongly self-dual polyhedra with at most eight vertices.}
\label{fig:ssdpoly}
\end{figure}

\begin{proof}

\noindent{\bf The number of vertices are  $n=4,5$ or $6$, respectively.} Clearly, the only ssd-polyhedron with four vertices is the regular tetrahedron inscribed into the unit sphere. Since, by observing the previous paragraph, there is no pyramidal ssd-polyhedron with an odd number of vertices for $n=5$, we have the only possibility that the polyhedron is simplicial. By the remark of the preceding paragraph, it is also impossible. Let the number of vertices be $n=6$. Then we have two combinatorial possibilities, either $\alpha_5=1$ and $\alpha_3=5$ or $\alpha_4=2$ and $\alpha_3=4$. The first combinatorial situation can be realised as a strongly self-dual pyramid. The second case is combinatorially unique; the two quadrangle faces $ABEF$ and $BCDE$, which must be on the common side $BE$. The six vertices form a six-angle $ABCDEF$ in the space, and two triangle faces are necessarily created by the edges at the vertex $B$ and $E$ (see Figure \ref{fig:sixvertices}). These are $ABC$ and $DFE$, respectively. The space-quadrangle $FACD$ is dissected by one of its edges into two triangle faces; for example, we assume that these are $ACF$ and $CFD$. Since $B$ and $E$ are in the quadrangle faces, $\sigma (B)$  and $\sigma (E)$ contain the edge $FC$, respectively. Since by Theorem \ref{thm:stronglysdpol}  for a vertex $v\not \in \sigma(v)$ we have $\sigma(B)=FCD$ and $\sigma(E)=ACF$. Hence, $\sigma (D)$ is $ABC$. On the other hand, $\sigma (F)$ contains $D$; consequently, $\sigma (D)$ contains $F$, which is a contradiction. Thus, there is no strongly self-dual polyhedron with this combinatorics.

\begin{figure}[ht]
\includegraphics[scale=1]{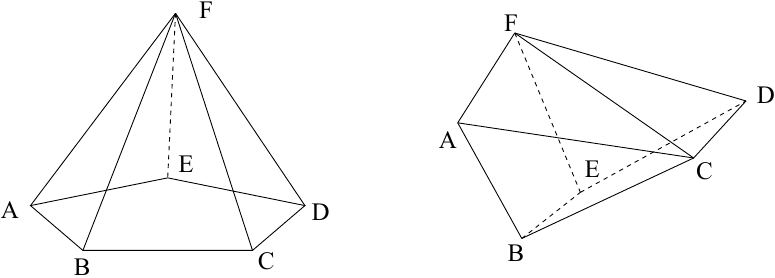}
\caption[]{Two combinatorial cases of six vertices}
\label{fig:sixvertices}
\end{figure}

\noindent{\bf The number of vertices is $n=7$.} Excluding the pyramid that cannot be strongly self-dual, we have two possibilities: $\alpha_5=1$ and $\alpha_4+\alpha_3=6$ or $\alpha_4+\alpha_3=7$. The first case implies that $\alpha_5=1$, $\alpha_3=5$ and $\alpha_4=1$; and the second one possible if $\alpha_3=4$ and $\alpha_4=3$. (If $\alpha_5=1$ and $\alpha_4\geq 2$, the number of vertices is greater than seven. On the other hand, if $\alpha_i=0$ for $i\geq 5$, then by equality \ref{eq:main} $\alpha_3=4.$)

In the first case, there is a unique combinatorial figure. Let $A, B, C, D, E$ be the respective vertices of the pentagon. Since the polyhedron has only two additional vertices, the quadrangle faces have a common side with this pentagon (e.g. $AB$). The valence of an additional vertex (e.g. of $F$) is five. Then, the valence of the last vertex $G$ has to be three because the dual face of a vertex cannot contain the vertex. We must also connect the vertex $F$ with the vertices $C, D$ and $E$. Finally, the last edge of the graph has to join $G$ and one of the vertices from $C, D$ and $E$. Assume that this edge is $GE$. (See the a) picture in Figure \ref{fig:sevenvertices}.) Hence, we have $\sigma(E)=ABFG$. The dual of $A$ and $B$ contain the vertices $E$ and $F$, so they also contain the edge $EF$. Hence, $EGF$ and $EFD$ cannot be the dual of the three-valence vertices $C$, $D$, and $G$. Hence $\sigma(C)=AEG$ moreover $\sigma(D)=BCF$ and $\sigma(G)=CFD$. But in this case, the dual of the edge $EG$ is the intersection of the dual of $E$ and $G$, which is impossible because $\sigma(E)\cap\sigma(G)=ABFG\cap CFD=F$ is not an edge. There is no ssd-polyhedron with such a combinatorial structure. Similar argument holds when one of the segments $GC$ and $GD$ is an edge.

\begin{figure}
\includegraphics[scale=1]{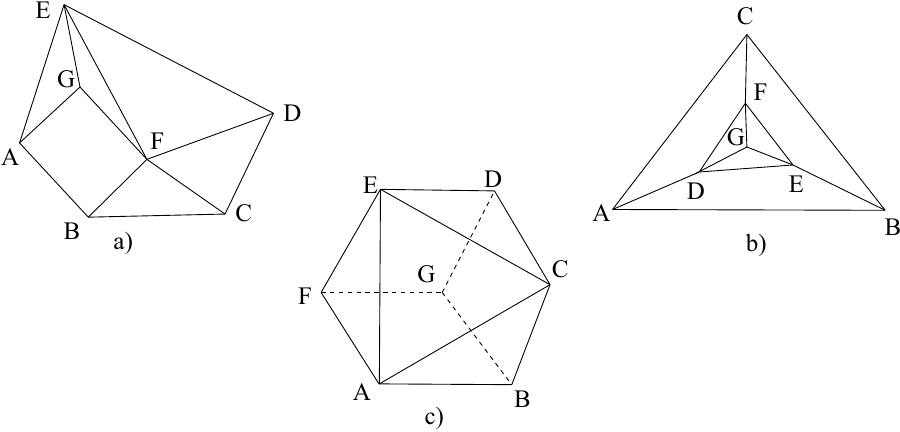}
\caption[]{Combinatorial cases of seven vertices}
\label{fig:sevenvertices}
\end{figure}

The second case has two subcases. The three quadrangles may contain one common edge pairwise, but these three faces have no common point. This gives an existing situation by Lov\'asz's construction (see in b) picture in Figure \ref{fig:sevenvertices}). The final combinatorial possibility is when the three quadrangle faces have a common vertex. Since there is no vertex with either degree five or degree two, duality forces the four triangle faces, as shown in c) picture in Fig. \ref{fig:sevenvertices}. We can assume that $\sigma(C)=EFGD$, $\sigma(A)=BCGD$ and $\sigma(E)=ABGF$. Since $\sigma(F)$ contains $E$ and $C$ it is equal to $ECD$. Similarly, $\sigma(D)=ABC$ and $\sigma(B)=AFE$. From this $FE=FD=FC$, $DC=DB=DA$ and $BA=BF=BE$. But $\sigma(E)$ contains $F$ and $B$, hence it is $ABGF$, meaning that $EF=EA=EB=EG$. Similarly $\sigma(C)=EDGF$, $\sigma(A)=BGDC$ implying $CD=CE=CF=CG$ and $AB=AC=AD=AG$. Hence, all these sides are equal; in fact, $FE=FD=FC=CE=CG=CD=DB=DA=AC=AG=AB=BF=BE=EG=EA=EF$. Hence, $EACG$ is a regular tetrahedron. Since $\sigma(C)=EDGF$ is orthogonal to the line through the origin of this tetrahedron from the vertex $C$, its plane is $\mathrm{aff}\{E, A, G\}$. This is a contradiction because the five points $D, E, F, G, A$ cannot be in the same plane. There is no strongly self-dual polyhedron in this combinatorial class.

\noindent{ \bf The number of vertices is $n=8$.} The combinatorial possibilities are:
\begin{enumerate}
\item[{\bf (i)}] $\alpha_7=1$, $\alpha_3=7$;
\item[{\bf (ii)}] $\alpha_6=1$, $\alpha_3+\alpha_4+\alpha_5=7$;
\item[{\bf (iii)}] $\alpha_5=1$, $\alpha_3+\alpha_4=7$;
\item[{\bf (iv)}] $\alpha_3+\alpha_4=8$.
\end{enumerate}

Since $2|E|=\sum_{l\geq 3} \alpha_l$, if the sum is odd, then there is no convex polyhedron corresponding to the sequence $\alpha_3,\alpha_4,\ldots$. Hence, we can exclude from the investigation the following cases:
\begin{itemize}
\item $\alpha_6=1, \alpha_3=7$;
\item $\alpha_5=1, \alpha_4=1, \alpha_3=6$;
\item $\alpha_4=3, \alpha_3=5$;
\item $\alpha_4=1, \alpha_3=7$.
\end{itemize}

\begin{figure}[ht]
\includegraphics[scale=0.8]{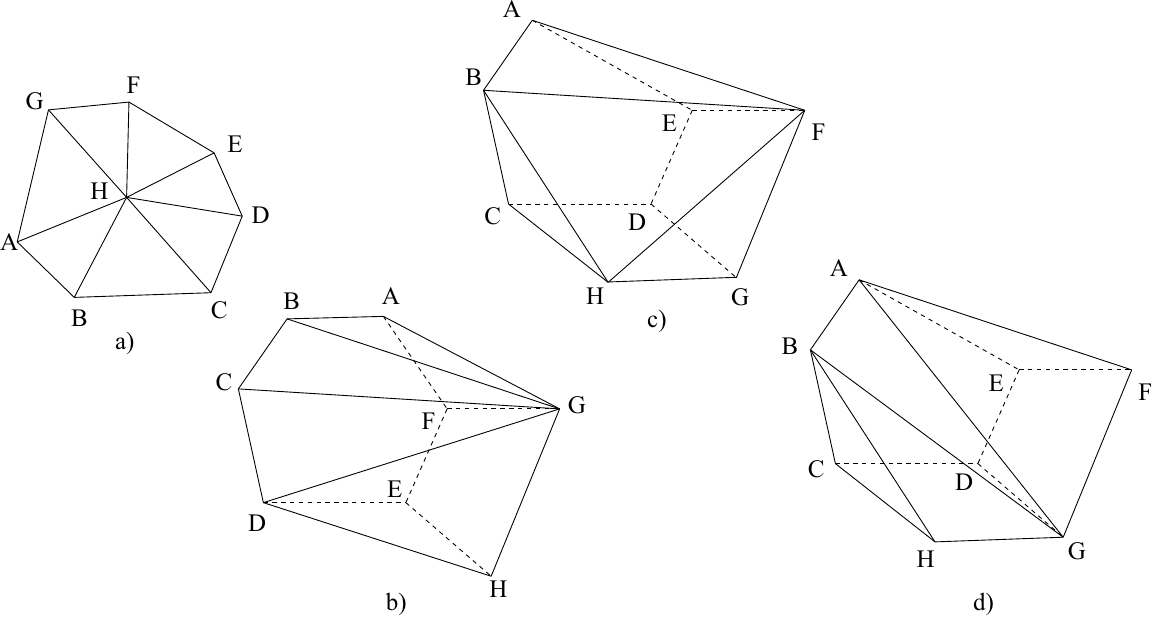}
\caption[]{Eight vertices I}
\label{fig:eightvertices1}
\end{figure}

\noindent {\bf Case  (i)} is realized by an $L$-type strongly self-dual pyramid (see a) in Figure \ref{fig:eightvertices1}).

\noindent {\bf In case (ii)}, we have the possibility that $\alpha_5=0$, $\alpha_4=1$ and $\alpha_3=6$. This situation can be seen in b) picture of \ref{fig:eightvertices1}. $GA$ and $HD$ should be edges, determining the triangle faces $AFG$ and $DEH$, respectively. Since one vertex among $G$ and $H$ has valence six, three new edges must go from this; for example, these connect $G$ to the vertices $B, C$ and $D$. Hence, the valence of $D$ is four. $\sigma(E)$ and $\sigma(F)$ contain $D$ and $G$, furthermore $\sigma(H)$ contains $D$. Since $\sigma(H)$ cannot be $EDH$, $GDH$ or $ABCDEF$ it is $GDC$ and hence one of the duals $\sigma(E)$ and $\sigma(F)$ does not exist.

\noindent{\bf In case (iii)}, from equality \ref{eq:main} we get then $\alpha_3=5$, hence $\alpha_4=2$. 
The three non-triangular faces have to be a common vertex. Denote by $A, B, C, D, E, F, G, H$ the respective vertices as in the c) picture in Fig. \ref{fig:eightvertices1}. We have two combinatorially distinct possibilities for the choice of the $5$-valence vertex. 
\begin{itemize}
    \item First choose the point $F$ to be the $5$-valence vertex. It is incident with one of the quadrangle faces e.g. $EDGF$, see the c) picture in Fig. \ref{fig:eightvertices1}. Clearly, $FA$, $FH$ and $BH$  should be edges. Hence, the last edge originated from $F$ has to connect $F$ with $B$ because three faces should not be incident with the edge $CH$. Now $H$ and $B$ are $4$-valence vertices and $\sigma(D)=BHF$. $\sigma(E)$ contains $F$ and $H$, but it can not be $BHF$, hence it is $FGH$. Similarly, $\sigma(C)=ABF$ and $\sigma(G)$ also contains $F$ and $B$ implying that it is $BCH$. Since $\sigma(A)$ contains $F$ and does not contain $A$, it has no dual, giving a contradiction.
    \item The second possibility, when the $5$-valence vertex is the common vertex $G$ of the two quadrangle faces. (See picture d, in Fig. \ref{fig:eightvertices1}.) This possibility can be realized as an ssd-polyhedron; see the $C_1$ shape of Table 1 in the manuscript \cite{katz-memoli-wang}. (The corresponding pointwise extremal configuration has spherical diameter $2,087071$, which is less than $2\pi/3$, and so the remark after Theorem 1 in \cite{katz} says that its convex hull is an ssd-polytope.)
\end{itemize}

\begin{figure}[ht]
\includegraphics[scale=1]{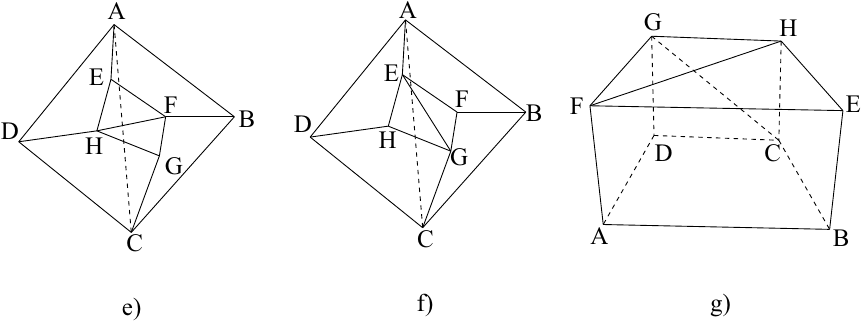}
\caption[]{Eight vertices II}
\label{fig:eightvertices2}
\end{figure}

{\bf In case (iv)}, $\alpha_4+\alpha_3=8$, so equality (\ref{eq:main}) implies that $\alpha_3=\alpha_4=4$. $\alpha_4=\alpha_3=4$. The four quadrangle faces have to create an edge-connected chain because the number of vertices of the polyhedron is eight. This can be realized in three ways.

\begin{itemize} 
\item In the first one, the elements of the closed chain successively join each other in an edge, and there is no common vertex of three from them. In this case, the four triangle faces must give the two "caps" of the polyhedron separated by the zone of quadrangles. Of course, the role of the two new edges of the cap is symmetric, and we have to investigate only the two non-symmetric cases, namely that either the vertices of valence four are $A$,$C$, $F$, and $H$, respectively, or $A$, $C$, $E$, $G$. (The e) and f) pictures in Fig. \ref{fig:eightvertices2}, respectively.) In the first case, if $\sigma(A)=BCGF$ then $\sigma(F)$ is either $CDHG$ or $ADHE$. These two possibilities are symmetric, so we have to investigate only one of them. In the first situation, $\sigma(H)=EFBA$ and $\sigma(C)=ADHE$ giving a contradiction, because $\sigma(E)$ contains $H$ and $C$, implying that $\sigma(F)=\sigma(E)$. If $\sigma(A)=GCDH$ by symmetry, we get a contradiction, similarly, too.

Let us assume that the vertices of valence four are $A$, $C$, $E$, and $G$. If again $\sigma(A)=BCGF$ then $\sigma(E)=CDHG$. If $\sigma(C)=AEFB$ then $\sigma(G)=AEHD$. Since $F$ is a common vertex of $BCGF$ and $BFEA$ then $\sigma(F)$ contains $A$ and $C$. $\sigma(B)$ also contains these points, hence the only possibility $\sigma(F)=ACB$ and $\sigma(B)=ACD$. The dual of the edge $GF$ has to be the intersection of $\sigma(G)=AEHD$ and $\sigma(F)=ACB$, which is not an edge of the polyhedron. Hence we get that $\sigma(C)=AEHD$ and $\sigma(G)=AEFB$. Since $H$ is a common vertex of $\sigma(C)$ and $\sigma(E)$ then $\sigma(H)$ contains the points $C$ and $E$. But now, a face contains these two vertices, creating a contradiction. Hence, there is no ssd-polyhedron with such a face lattice.

\item When three quadrangle faces meet in a vertex we get an edge-vertex graph that can be seen in the right picture of Figure \ref{fig:eightvertices2}. $\sigma(F)$ could be either $ABCD$ or $BCHE$. If $\sigma(F)=ABCD$, then $\sigma(A)$, $\sigma(B)$ and $\sigma(D)$ are triangular faces containing the vertex $F$. Since through $F$ there are only two triangular faces, we get that $\sigma(F)=BCHE$. If $\sigma(C)$ would be $ABEF$ then the triangular faces $\sigma(A)$, $\sigma(B)$ and $\sigma(E)$ contains the vertex $C$, which also impossible. Thus $\sigma(C)=ADGF$. Hence either $\sigma(G)=ABEF$ and $\sigma(H)=ABCD$ or vice versa. These two cases are equivalent to each other. We investigate the situation when $\sigma(G)=ABEF$ and $\sigma(H)=ABCD$. Consider the three-valence vertices. Their duals are triangular faces. In $E$, join three quadrangle faces; its dual is a triangle with four-valence vertices. Thus $\sigma(E)$ is either $FGH$ or $CGH$. Since $GH$ is a common edge of these possible faces in both cases the face $\sigma(H)$ should contain the vertex $E$. But we assumed that $\sigma(H)=ABCD$ which leads to a contradiction. Hence, this combinatorial case excludes the possibility of strong self-duality, as we stated.  

\item Suppose that the four quadrilaterals have a common vertex. Then, we find a contradiction in the fact that no vertex is contained in its dual face.

\end{itemize}
We have now considered all cases and obtained two possible configurations. The first, when $\alpha_7=1, \alpha_3=7$ can be constructed with Lov\'asz's method, and the second one can be implemented with parameters $\alpha_5=1, \alpha_4=2$ and $\alpha_3=5$ by Katz, Memoli and Wang's method. 
\end{proof}

\begin{remark} As a checking of the existence of the Katz-Memoli-Wang's polyhedron we construct it with our independent method of the next section. We found that it realizes with the parameters 
$$
r = 0.493643648472824, \quad 
\kappa = 25.73186609765885^\circ, \quad 
\lambda = 167.1340669511706^\circ.
$$
We also made a metric view of this polyhedron which can be seen in Figure \ref{fig:katz-memoli-wang}
\end{remark}

\begin{figure}[ht]
\includegraphics[scale=0.25]{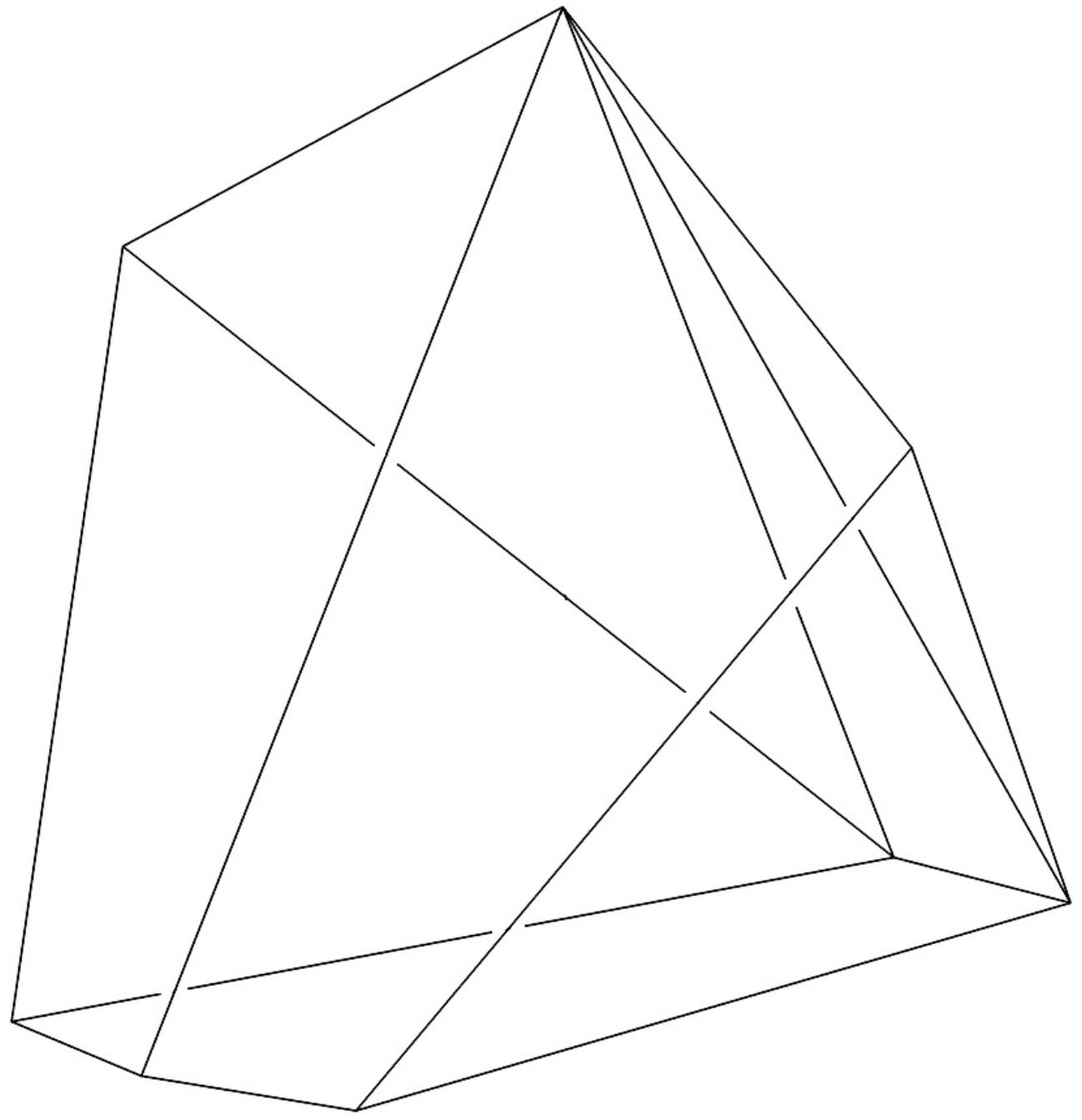}
\caption[]{The Katz--Memoli--Wang's ssd-polyhedron with eight vertices.}
\label{fig:katz-memoli-wang}
\end{figure}

\section{The construction of a non-$L$-type ssd-polyhedron}\label{sec:construction}

Using the algorithm of Theorem \ref{thm:determination}, we construct a new ssd-polyhedron which is combinatorially non-isomorphic to an $L$-type polyhedron. 

\subsection{Formulas on the dual-edges} \label{subsec:dualedges}
Consider the edge $AB$ of the ssd-polyhedron. The duality map $\sigma$ sends $A$ to the face $\sigma(A)$. The circumscribed circle $c$ of $\sigma(A)$ has the centre $K$, where $K$ is the intersection of the line $AO$ with the inscribed sphere of the polyhedron of radius $r$. The  circle $c$ lies on the circumscribed sphere of the polyhedron and has radius $\rho=\sqrt{1-r^2}$. The plane of $c$ is orthogonal to the line $AO$. Similarly, the centre of the circumscribed circle of the face corresponding to the vertex $B$ is the intersection point $L$ of the line $BO$ with the inscribed sphere of the polyhedron. The dual edge $XY$ of the edge $AB$ belongs to both of the above planes; hence, its endpoints $X$ and $Y$ are the intersection points of the two circles. Hence, the lines $AB$ and $XY$ are orthogonal, skew lines whose normal transversal connects the respective edges' midpoints through the centre $O$. Thus, the position vectors $x$ and $y$ (with respect to the origin $O$) of the points $X$ and $Y$ can be determined by the vectors $a$ and $b$ of the points $A$ and $B$ as follows:

\begin{figure}[ht]
   \begin{center}
   \begin{minipage}{0.48\textwidth}
     \includegraphics[width=.7\linewidth]{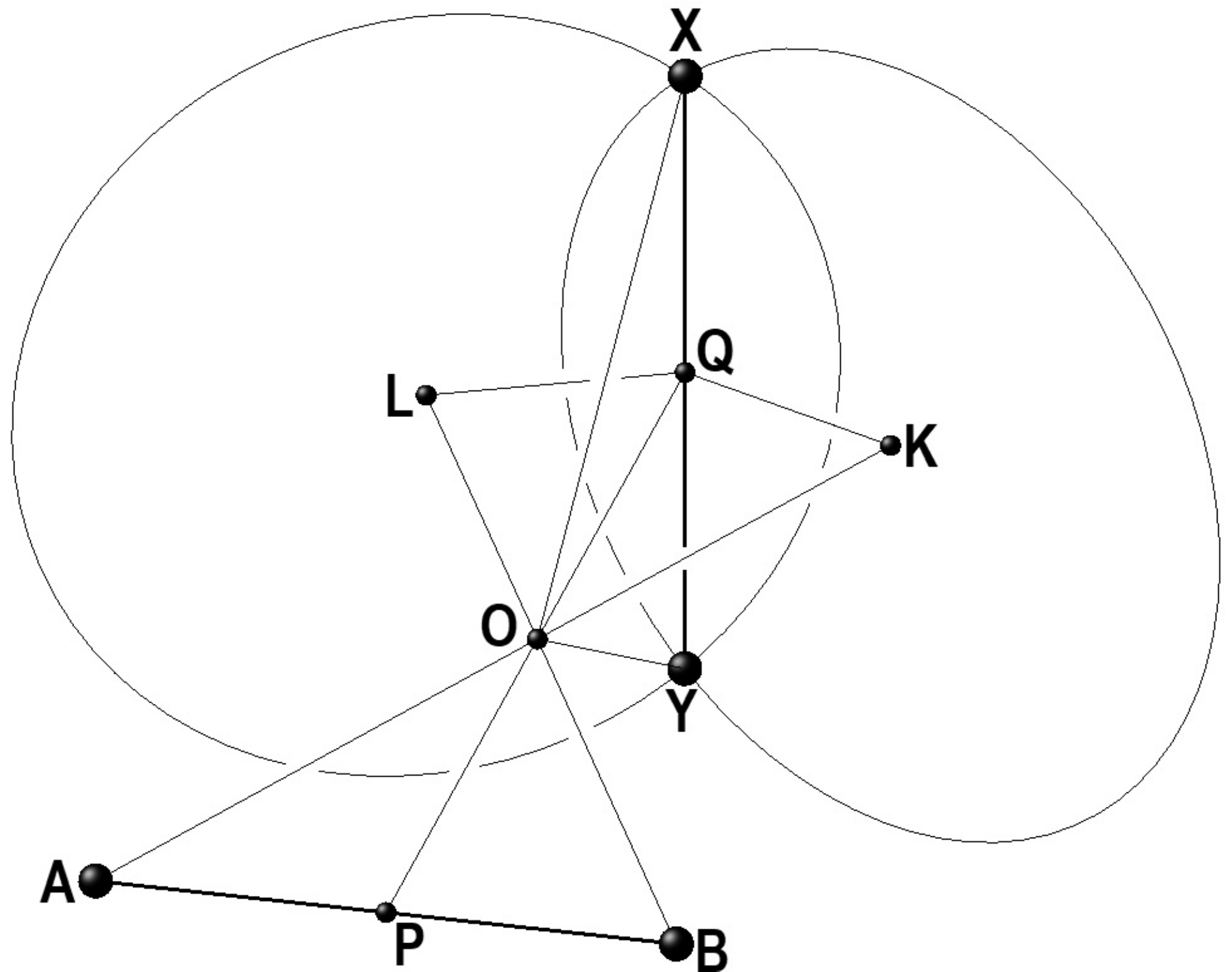}
     \caption{Dual edges I}\label{fig:dual_edges_1}
   \end{minipage}
   \begin {minipage}{0.48\textwidth}
     \includegraphics[width=.7\linewidth]{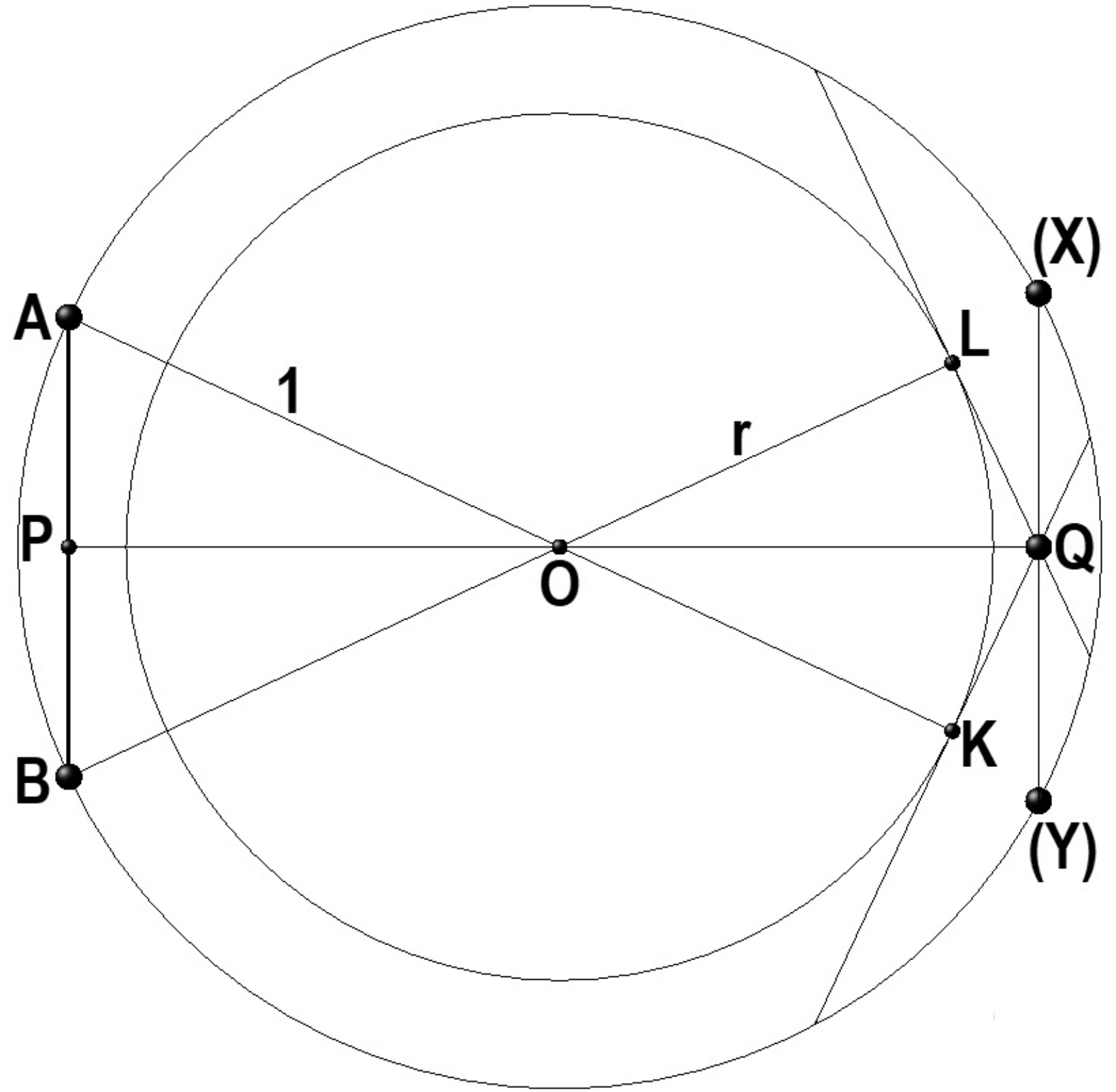}
     \caption{Dual edges 2}\label{fig:dual_edges_2}
   \end{minipage}
   \end{center}
\end{figure}

Since $a$ and $b$ are unit vectors, we have
\begin{equation}
|a+b|=\sqrt{2}\sqrt{1+\langle a,b\rangle}; \, |a-b|=\sqrt{2}\sqrt{1-\langle a,b\rangle}; \, |a\times b|=\sqrt{1+\langle a,b\rangle}\sqrt{1-\langle a,b\rangle}, 
\end{equation}
from which we get by the notation of Figure \ref{fig:dual_edges_1} and \ref{fig:dual_edges_2} that
\begin{equation}
|\overrightarrow{OP}|=|p|=\frac{1}{\sqrt{2}}\sqrt{1+\langle a,b\rangle} \mbox{ and } |\overrightarrow{PA}|=\frac{1}{\sqrt{2}}\sqrt{1-\langle a,b\rangle},
\end{equation}
respectively. (Here, we must consider that $r<|OP|<1$.)
Since the triangles $APO$ and $QKO$ are similar to  each other, we have that $|q|:r=1:|p|$
\begin{equation}
q=\overrightarrow{OQ}=-\frac{r}{1+\langle a,b\rangle}(a+b) \mbox{ and } |q|=\frac{\sqrt{2}r}{\sqrt{1+\langle a,b\rangle}}.
\end{equation}
On the other hand, $\overrightarrow{QX} ||(a \times b)$ and $|\overrightarrow{QX}|^2+|\overrightarrow{OQ}|^2=1$ we get
\begin{equation}
|\overrightarrow{QX}|= \frac{\sqrt{1+\langle a,b\rangle -2r^2}}{\sqrt{1+\langle a,b\rangle}} \mbox{ implying that } \overrightarrow{QX}=\frac{1}{1+\langle a,b\rangle}\sqrt{\frac{1+\langle a,b\rangle -2r^2}{1-\langle a,b\rangle}}(a \times b).
\end{equation}
Similarly, we get for $Y$ that
\begin{equation}
\overrightarrow{QY}=\frac{1}{1+\langle a,b\rangle}\sqrt{\frac{1+\langle a,b\rangle -2r^2}{1-\langle a,b\rangle}}(b \times a),
\end{equation}
and finally, we can introduce the vector-valued map $\Phi_r(a,b)$ also depends on the parameter $r$ by the formula:
\begin{equation}\label{eq:defofPhi}
\Phi_r(a,b):= \frac{1}{1+\langle a,b\rangle}\left(\sqrt{\frac{1+\langle a,b\rangle -2r^2}{1-\langle a,b\rangle}}(a \times b)-r(a+b)\right).   
\end{equation}
Now we get $\Phi_r(a,b)=x$; $\Phi_r(b,a)=y$ moreover $a=\Phi_r(x,y)$ and $b=\Phi_r(y,x)$ also hold.

\subsection{Searching for a non-L-type ssd-polyhedron}

Theorem \ref{thm:determination} shows that one of its faces uniquely determines an $L$-type ssd-polyhedron. The steps of constructing the proof can be used originating from an arbitrary face of a polyhedron inscribed in the unit sphere, and we can investigate whether the algorithm gives or not an ssd-polyhedron. By adding some further assumptions, we think the above theorem can be generalized to a broader class of polyhedra. To verify our idea, we consider a non-regular pentagon as a face and try to construct an ssd-polyhedron originating from this face. If we are successful, then the polyhedron is not an L-type one since if an L-type polyhedron has a pentagonal face, then it should be regular.  
Consider Figure \ref{fig:Figure0}, where we can see a pentagon with two right angles, where $\kappa=\kappa_0=45^\circ$ and $\lambda=\lambda_0=135^\circ$.
\begin{figure}[ht]
     \begin{center}
   \begin{minipage}{0.48\textwidth}
     \includegraphics[width=.7\linewidth]{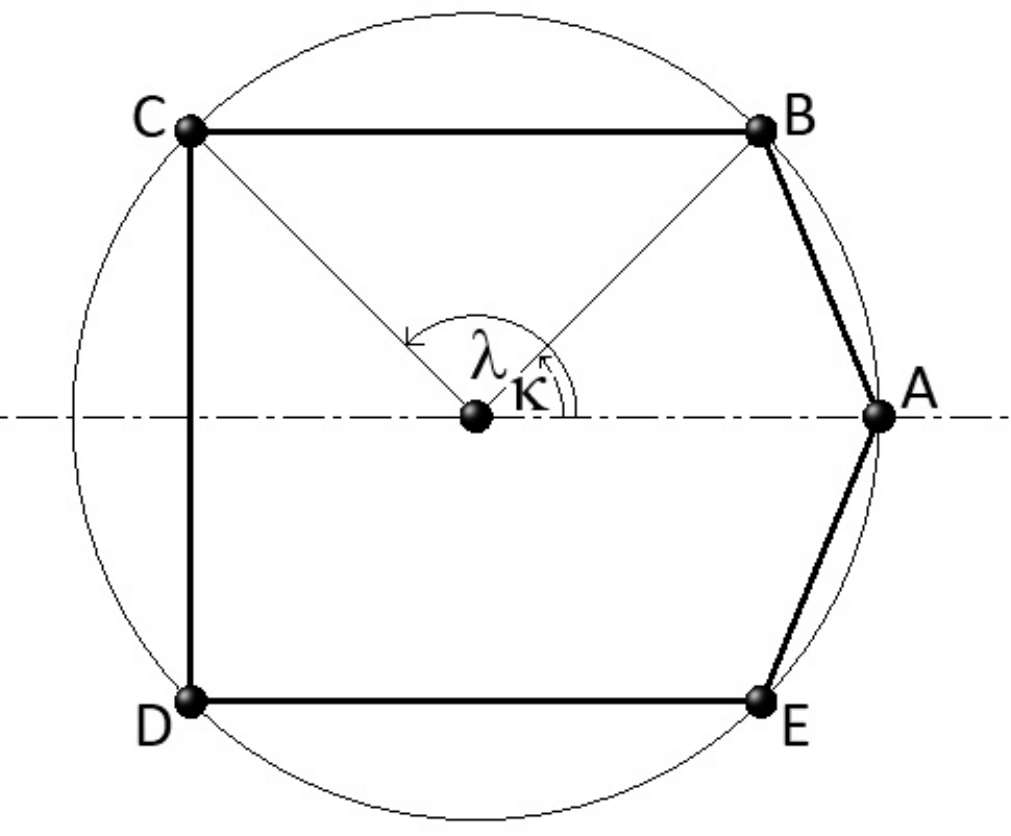}
     \caption{The original pentagon}\label{fig:Figure0}
   \end{minipage}
   \begin {minipage}{0.40\textwidth}
     \includegraphics[width=.7\linewidth]{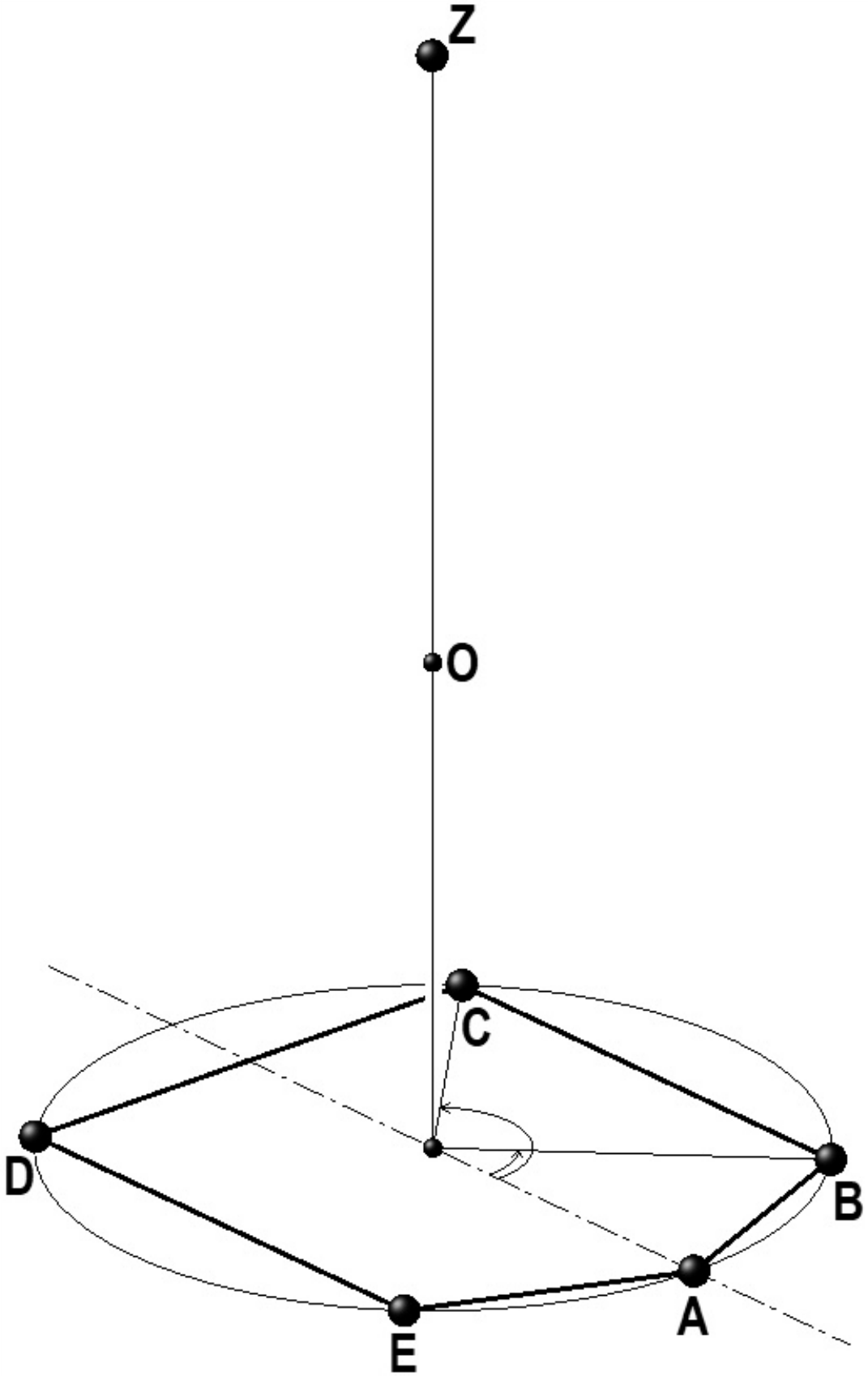}
     \caption{The first duality}\label{fig:Figure1}
   \end{minipage}
   \end{center}
\end{figure}

The lengths of the sides are determined by the radius $r$ of the inscribed sphere. First, we assume that it is $r=r_0=0.8$. Hence, the radius of the circumscribed circles of the faces (and therefore the radius of the circle of the face $ABCDE$) is $\rho=\sqrt{1-r^2}=0.6$. Now the point $O$ is on the line, perpendicular to the face $ABCDE$ and $r=0.8$ far from the center of the circle. The duality map corresponds to this face a vertex $Z$ whose distance from the center of the circle is $1+r=1.8$ (see Figure \ref{fig:Figure1}).

The first non-trivial step of the process is to determine the dual of the sides of the basic face. For example, the vertices of the dual of the edge $CD$ can be obtained by the intersection of the circles whose centres are on the line $CO$ and $DO$, respectively and whose planes are orthogonal to these respective lines. In Figure \ref{fig:Figure2} we can see the realization of this step and the respective duals of the edges of the face $ABCDE$. These are denoted by $ZF, ZG, ZH, ZI$ and $ZJ$, respectively. 
\begin{figure}
    \centering
    \includegraphics[width=0.5\linewidth]{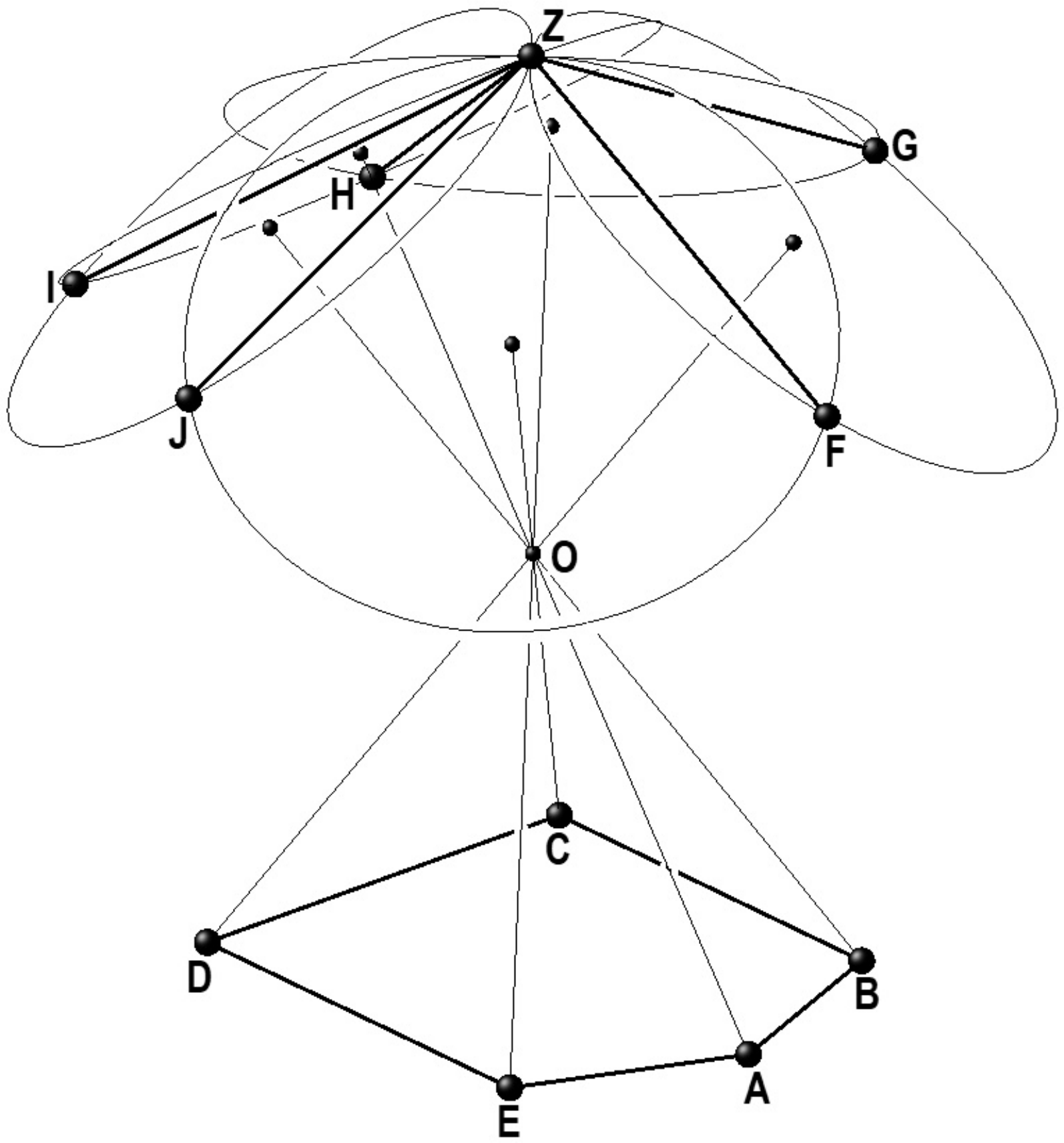}
    \caption{The first step of the construction.}
    \label{fig:Figure2}
\end{figure}
By the map $\Phi_r(\cdot,\cdot)$ this step can be written by formulas, too.
\begin{equation}\label{eq:firststep}
f=\Phi_r(d,c) \quad  g=\Phi_r(e,d) \quad h=\Phi_r(a,e) \quad i=\Phi_r(b,a) \quad j=\Phi_r(c,b) 
\end{equation}
and also
$$
z=\Phi_r(c,d)=\Phi_r(d,e)=\Phi_r(e,a)=\Phi_r(a,b)=\Phi_r(b,c).
$$
From the new vertices $F, G, H, I, J$ we again construct the dual faces and the possible third edges from the face's vertices $A, B, C, D, E$. The corresponding new vertices are $K, L, M, N, P$, the respective edges are $AK, BL, CM, DN$ and $EP$; as shown in Figure \ref{fig:Figure3}. This figure also contains the duals of the new sides, which are the segments $HI, IJ, JF, FG$ and $GH$, respectively. 
\begin{figure}[ht]
     \begin{center}
   \begin{minipage}{0.48\textwidth}
     \includegraphics[width=.7\linewidth]{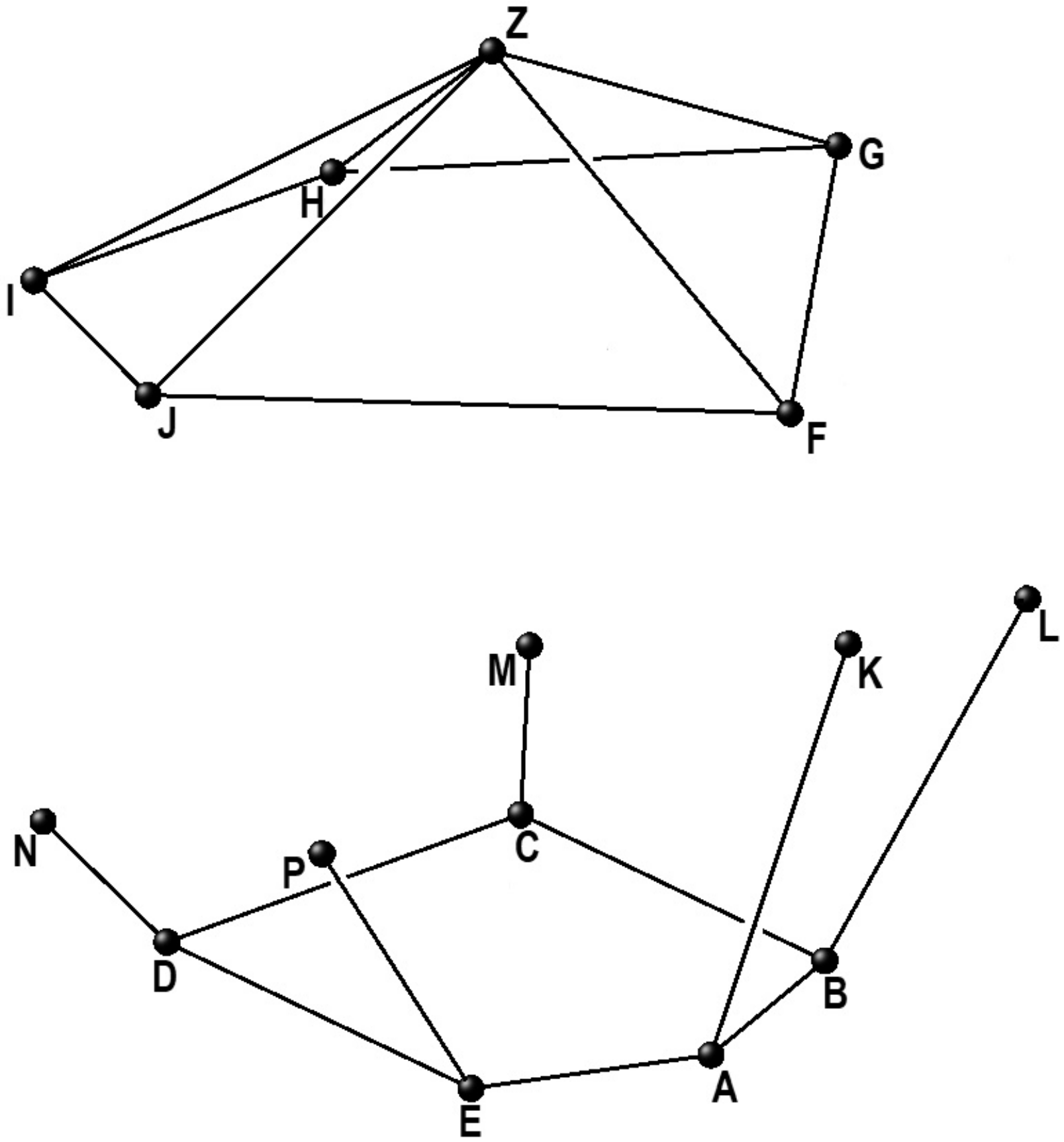}
     \caption{The second step of the construction.}\label{fig:Figure3}
   \end{minipage}
   \begin {minipage}{0.48\textwidth}
     \includegraphics[width=.7\linewidth]{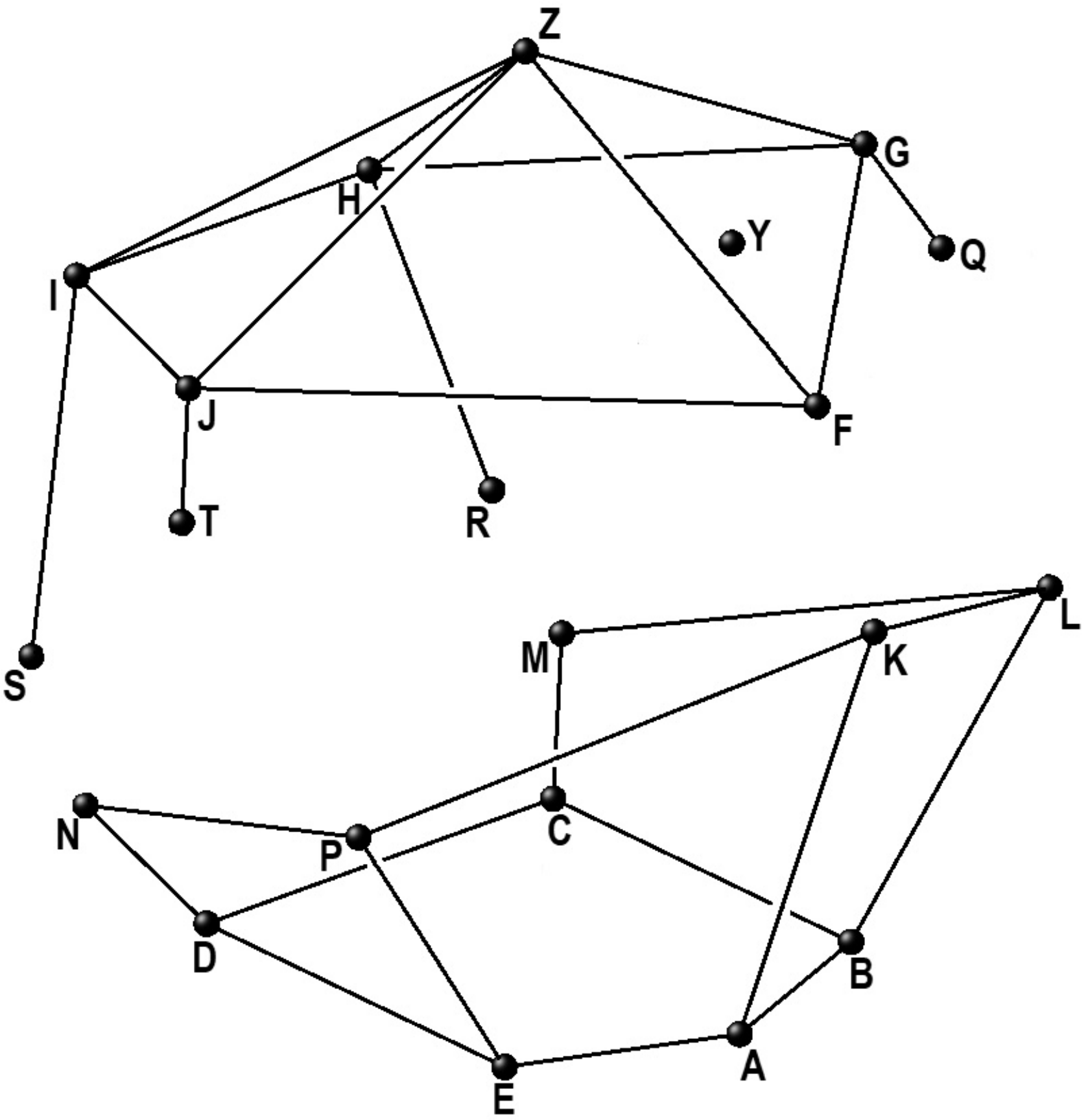}
     \caption{The third step of the construction}\label{fig:Figure4}
   \end{minipage}
   \end{center}
\end{figure}

The position vectors of the new vertices are:
\begin{equation}\label{eq:secondstep}
k=\Phi_r(h,i) \quad l=\Phi_r(i,j) \quad m=\Phi_r(j,f) \quad n=\Phi_r(f,g) \quad p=\Phi_r(g,h).
\end{equation}

We use the vertices $K, L, M, N, P$ in the next step. We get the edges $GQ, HR, IS$ and $JT$, and their duals $NP, PK, KL, LM$. We also get the point $Y$ as a possible vertex as the intersection point of the circles corresponds to the dual faces of the vertices $M$ and $N$. This point from the origin $O$ is behind the edge $ZF$, so it shouldn't create a new edge for our convex polyhedron. Hence, $FY$ and its dual segment $MN$ are not edges of the polyhedron. At this step, we have only four proper new vertices:

\begin{equation}\label{eq:thirdstep}
q=\Phi_r(p,n) \quad r=\Phi_r(k,p) \quad s=\Phi_r(l,k) \quad t=\Phi_r(m,l), 
\end{equation}
and one non-proper vertex $y=\Phi_r(m,n)$.

Finally, we construct from the vertices $Q, R, S$ and $T$. We get the edges $LU$ and $PV$, and their duals $ST$ and $QR$ (see in Figure \ref{fig:Figure5}). We also see that the dual circles of the vertices $R$ and $S$ intersect at the points $K$ and $X$, where $X$ cannot be a proper vertex, as in the case of $Y$. The three new vertices are the proper 
\begin{equation}\label{eq:fourthstep}
u=\Phi_r(s,t) \quad v=\Phi_r(q,r) 
\end{equation}
and the non-proper one $x=\Phi_r(r,s)$.

But $X$ is very close to the point $F$, and by the assumption $F=X$ the process combinatorially closing, the existence of the edge $KX=KF$ implies the existence of its dual $RS$. Now, the circumcircle of the dual face of the last vertex $U$ contains the points $S$ and $T$ and is very close to the vertices $N$ and $V$. If we assume that the space-quadrangle $STVN$ is a combinatorial quadrangle face of the searched polyhedron, then we also have that $NPV$ is a triangle face whose dual is the vertex $Q$. Using the plane symmetry of the construction we also assume that the combinatorial dual of the vertex $V$ is the space-quadrangle $QRMU$ and we also have a new triangle face $LMU$ whose dual is the vertex $T$.

\begin{figure}[ht]
     \begin{center}
   \begin{minipage}{0.48\textwidth}
     \includegraphics[width=.7\linewidth]{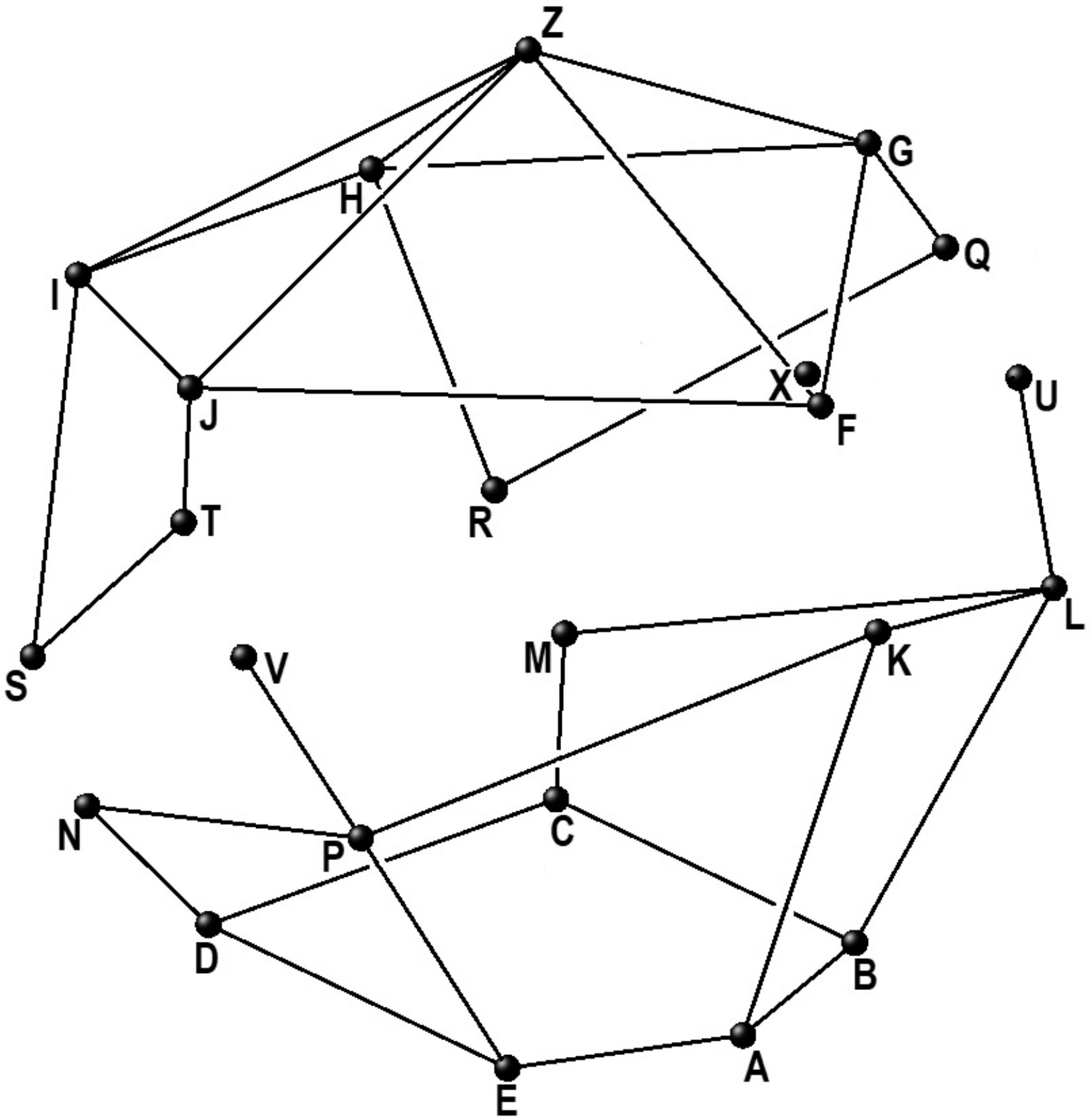}
     \caption{The fourth step of the construction.}\label{fig:Figure5}
   \end{minipage}
   \begin {minipage}{0.48\textwidth}
     \includegraphics[width=.7\linewidth]{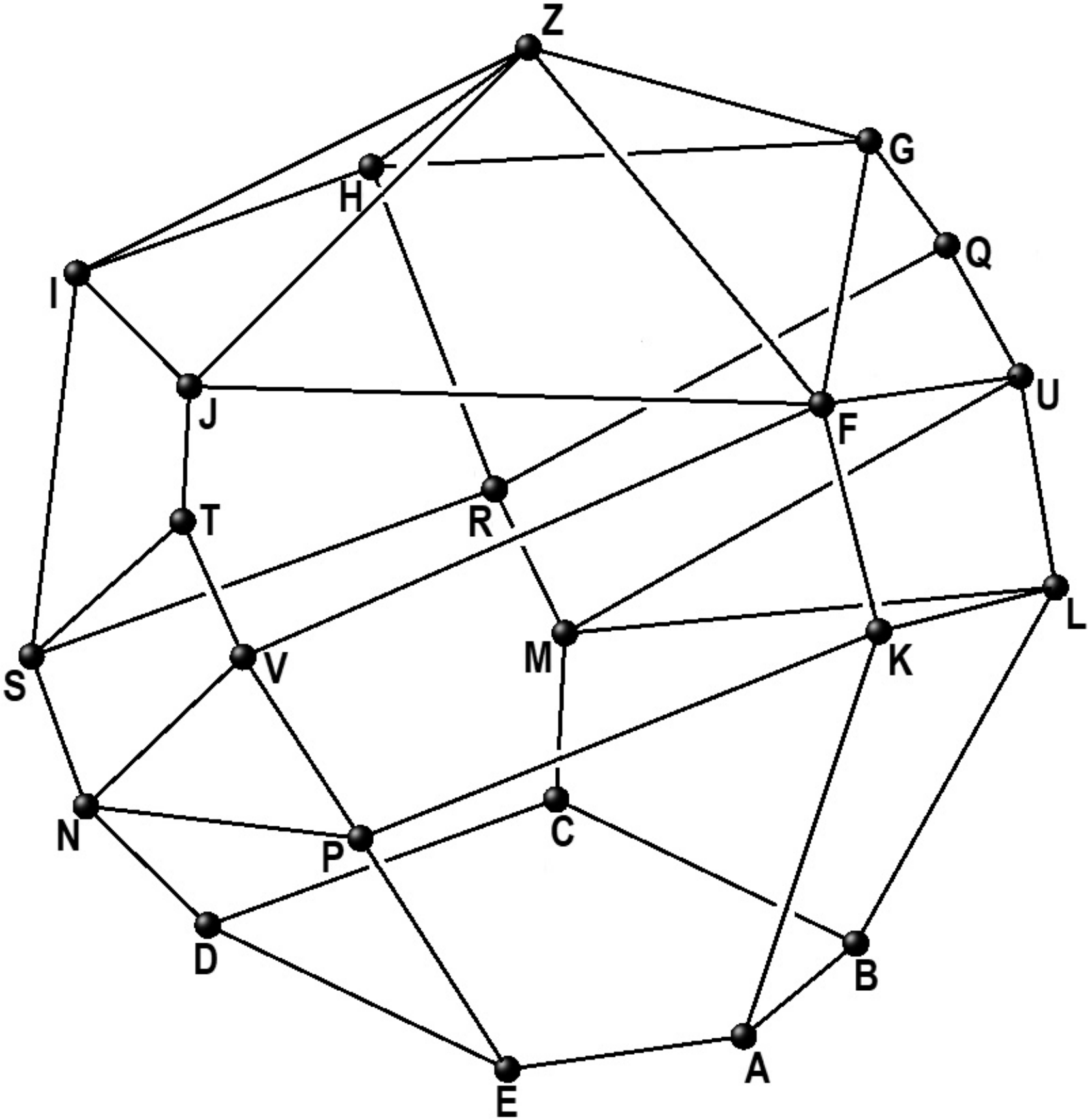}
     \caption{The fifth step of the construction}\label{fig:Figure6}
   \end{minipage}
   \end{center}
\end{figure}
Finally, for the whole combinatorial structure of the polyhedron, we have to draw the edges $MR$ and $NS$ and their duals $FU$ and $FV$ (see Figure \ref{fig:Figure6}). In this step, we create the hexagonal face $CDNSRM$, which is the dual of the vertex $F$; the deltoid-like $FJTV$ and $FUQG$ faces, which are the respective duals of the vertices $M$ and $N$ and the quadrangle faces $FVPK$ and $FULK$, which are the dual of the vertices $R$ and $S$.
\begin{figure}[ht]
     \begin{center}
        \includegraphics[width=.7\linewidth]{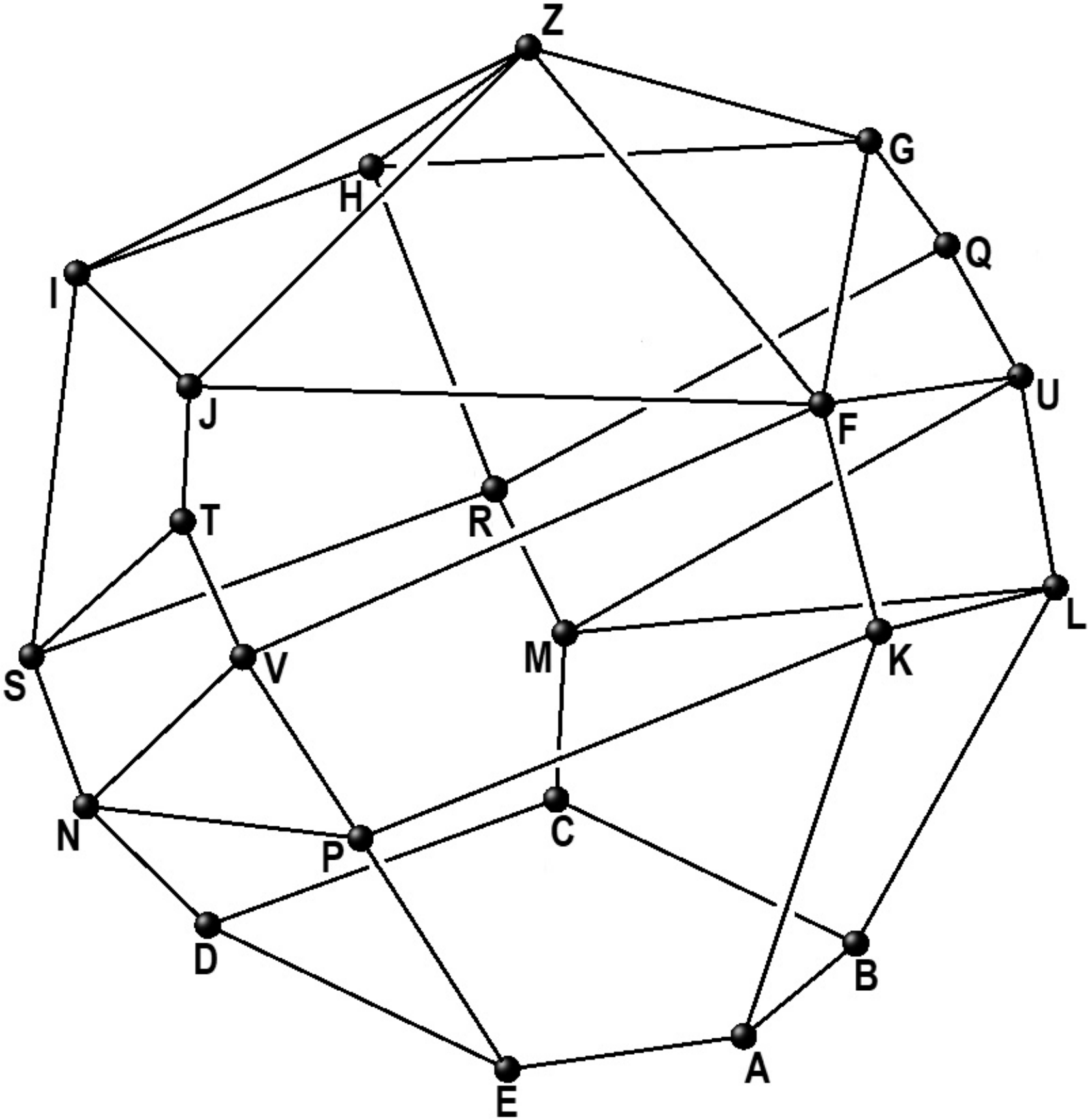}
     \caption{The combinatorial polyhedron.}\label{fig:combinatorialsolution}
   \end{center}
\end{figure}

In Figure \ref{fig:combinatorialsolution}, we drew these edges and faces to show that this combinatorics agrees with the ssd-property. To get a proper metric ssd-polyhedron, we must choose the parameters so that all faces will be planar. The faces which metrically doesn't appropriate are $CDNSRM$, $STVN$, $QRMU$, $FJTV$, $FUQG$, $FVPK$ and $FULK$. Note that if $F$ and $X$ agree, then the faces $CDNSRM$, $FVPK$ and $FULK$ are automatically planar, and the basic pentagon's symmetry implies the polyhedron's planar symmetry. Hence, we have to check the planarity the two quadrangle faces $STVN$ and $FJTV$. In both situations, we must check whether the vertex $V$ incident or not to the plane of the other three vertices, which are the dual planes of the respective vertices $U$ and $M$. We saw that all steps of the construction were uniquely determined by the original data $\kappa,\lambda $ and $r$. We search in a little neighbourhood of these data an exact solution of our task preserve the plane symmetry of the polyhedron at the plane $AOZ$. To this we have to solve the following equation system:
\begin{eqnarray}
    |f-x| & = & 0 \\
    \frac{(s-n)(t-n)(v-n)}{|(s-n)\times (t-n)|} & = & 0 \\
    \frac{(f-t)(j-t)(v-t)}{|(f-t)\times (j-t)|} & = & 0
\end{eqnarray}
where the left sides of the second and third equality mean the distance of the vertex $V$ to the planes $NST$ and $FJT$, respectively. The vector of the vertex $A$ depends on $r$ (it is $a(r)$), and the other four vectors depend on $r$ and one of the angles $\kappa$ and $\lambda$. These are $b(r,\kappa), c(r,\lambda), d(r,\lambda)$ and $e(r,\kappa)$, respectively. All other vertices depend on these vertices and the map $\Phi_r(\cdot,\cdot)$ is also depends on $r$. The dependence of the symbols on the original data can be written precisely:
$$
f=\Phi_r(d,c); g=\Phi_r(e,d); h=\Phi_r(a,e); i=\Phi_r(e,a); j=\Phi_r(c,e)
$$
$$
k=\Phi_r(\Phi_r(a,e),\Phi_r(b,a)); l=\Phi_r(\Phi_r(b,a),\Phi_r(c,b));
m=\Phi_r(\Phi_r(c,b),\Phi_r(d,c)); 
$$
$$
n=\Phi_r(\Phi_r(d,c),\Phi_r(e,d)); p=\Phi_r(\Phi_r(e,d),\Phi_r(a,e))
$$
$$
q=\Phi_r(\Phi_r(\Phi_r(e,d),\Phi_r(a,e)),\Phi_r(\Phi_r(d,c),\Phi_r(e,d)))
$$
$$
r=\Phi_r(\Phi_r(\Phi_r(a,e),\Phi_r(b,a)),\Phi_r(\Phi_r(e,d),\Phi_r(a,e)))
$$
$$
s=\Phi_r(\Phi_r(\Phi_r(b,a),\Phi_r(c,b)),\Phi_r(\Phi_r(a,e),\Phi_r(b,a)))
$$
$$
t=\Phi_r(\Phi_r(\Phi_r(c,b),\Phi_r(d,c)),\Phi_r(\Phi_r(b,a),\Phi_r(c,b)))
$$
\begin{flalign*}
u=\Phi_r(\Phi_r(\Phi_r(\Phi_r(b,a),\Phi_r(c,b)),\Phi_r(\Phi_r(a,e),\Phi_r(b,a))),&&
\end{flalign*}
\begin{flalign*} &&\Phi_r(\Phi_r(\Phi_r(c,b),\Phi_r(d,c)),\Phi_r(\Phi_r(b,a),\Phi_r(c,b))))
\end{flalign*}
\begin{flalign*}
v=\Phi_r(\Phi_r(\Phi_r(\Phi_r(e,d),\Phi_r(a,e)),\Phi_r(\Phi_r(d,c),\Phi_r(e,d))),&&
\end{flalign*}
\begin{flalign*} && \Phi_r(\Phi_r(\Phi_r(a,e),\Phi_r(b,a)),\Phi_r(\Phi_r(e,d),\Phi_r(a,e))))
\end{flalign*}
\begin{flalign*}
x=\Phi_r(\Phi_r(\Phi_r(\Phi_r(a,e),\Phi_r(b,a)),\Phi_r(\Phi_r(e,d),\Phi_r(a,e))),&&
\end{flalign*}
\begin{flalign*}&&
\Phi_r(\Phi_r(\Phi_r(b,a),\Phi_r(c,b)),\Phi_r(\Phi_r(a,e),\Phi_r(b,a))))
\end{flalign*}
Because of the complicated structure of the equation system, we search for a numerical solution. We use the following coordinates:
$$
a=(\rho, 0, -r); \, \, b=(\rho\cos \kappa, \rho\sin \kappa, -r);\, \, c=(\rho\cos\lambda, \rho \sin\lambda,-r)
$$
$$
d=(\rho\cos\lambda, -\rho \sin\lambda,-r); \, \, e= (\rho\cos\kappa, -\rho \sin\kappa,-r); \, \, z=(0, 0, 1),
$$
where  $\rho=\sqrt{1-r^2}$.
The process is as follows. Taking the values $\kappa_0=45^\circ$, $\lambda_0=135^\circ$ and $r_0=0.8$ as the base, consider the neighborhood of radius $\delta_0=0.1$ and divide these intervals into equal parts $n=200$. We get a three-dimensional grid for the parameter domains. We determine the polyhedron's vertices and the required three distances at all lattice points. The calculation error is the maximum value of the above three distances. We choose a lattice point $(\kappa_1,\lambda_1,r_1)$ with minimal error; this will be the origin of the next step. Parallelly, we restrict the radius of the examined domain by choice $\delta_1=\frac {\delta_0}{3}$. We divide the shorter intervals into $n$ parts again and densing the grid locally. Repeated the calculation of the first step, we choose a new point $(\kappa_2,\lambda_2,r_2)$ of the parameter domain and repeat the process. We stop when the error reduces below a given bound. In our case, required the bound $10^{-15}$ the twenty-seventh step gives the following result:
\begin{equation}\label{eq:finalresult}
    \kappa_{27}=45.18708115925679^\circ, \quad \lambda_{27}=137.9708898008123^\circ, \quad r_{27}=0.801257067121262.
\end{equation}
The vertices of the polyhedron (with the same calculation accuracy) are in the table below.

$$
\begin{tabular}{l}
$Z\,(0.0,\; 0.0,\; 1.0)$, \\
$A\,(0.5983202423353512,\; 0.0,\; -0.8012570671212621)$, \\
$B\,(0.4216926266320513,\; 0.4244554641330405,\; -0.8012570671212621)$, \\
$C\,(-0.4444351271090439,\; 0.4005802418739614,\; -0.8012570671212621)$, \\
$D\,(-0.4444351271090439,\; -0.4005802418739614,\; -0.8012570671212621)$, \\
$E\,(0.4216926266320513,\; -0.4244554641330405,\; -0.8012570671212621)$, \\
$F\,(0.8483424447791927,\; 0.0,\; 0.5294479165943166)$, \\
$G\,(0.0224329604142071,\; 0.8138064392649369,\; 0.5807028859046416)$, \\
$H\,(-0.8628874394036844,\; 0.3590712428534728,\; 0.3556586980168142)$, \\
$I\,(-0.8628874394036844,\; -0.3590712428534728,\; 0.3556586980168142)$, \\
$J\,(0.0224329604142071,\; -0.8138064392649369,\; 0.5807028859046416)$, \\
$K\,(0.9891443044532439,\; 0.0,\; 0.1469474224602405)$, \\
$L\,(0.4887101245345391,\; 0.8460369840318558,\; -0.2130348230400762)$, \\
$M\,(-0.6075422756722804,\; 0.5825733695380467,\; -0.5399080036228703)$, \\
$N\,(-0.6075422756722804,\; -0.5825733695380467,\; -0.5399080036228703)$, \\
$P\,(0.4887101245345391,\; -0.8460369840318558,\; -0.2130348230400762)$, \\
$Q\,(0.1228243012703047,\; 0.9324888687546015,\; 0.3396744039020674)$, \\
$R\,(-0.7680489592088464,\; 0.5746015768538119,\; -0.2827257047658043)$, \\
$S\,(-0.7680489592088464,\; -0.5746015768538119,\; -0.2827257047658043)$, \\
$T\,(0.1228243012703047,\; -0.9324888687546015,\; 0.3396744039020674)$, \\
$U\,(0.2797981540096859,\; 0.9490145065069506,\; 0.1452048878383263)$, \\
$V\,(0.2797981540096859,\; -0.9490145065069506,\; 0.1452048878383263)$, \\
$X\,(0.8483424447791927,\; 0.0,\; 0.5294479165943166)$. 
\end{tabular}
$$
 The orthogonal projections of the final polyhedron can be seen in Figure \ref{fig:Figure7}. The pictures show the three standard orthogonal Monge images and a further orthogonal projection in a general direction of the body. 

\begin{figure}[ht]
\begin{center}
        \includegraphics[width=.6\linewidth]{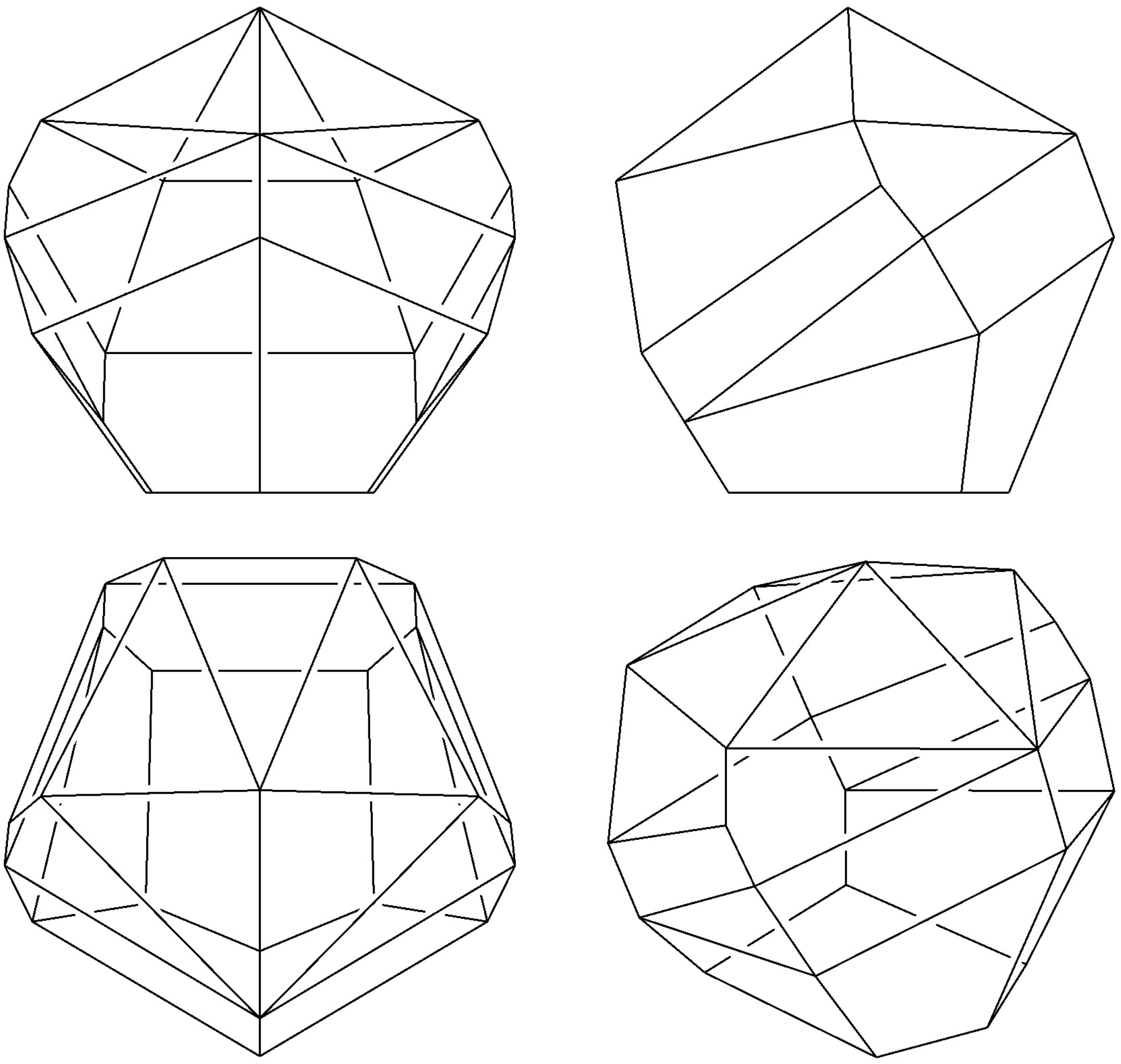}
     \caption{The othogonal projections of the new non-L-type ssd- polyhedron}\label{fig:Figure7}
   \end{center}
\end{figure}


\begin{thebibliography}{99}

\bibitem{dillencourt}  Dillencourt, M. B.: Polyhedra of Small Order and Their Hamiltonian Properties {\em Journal of combinatorial theory, Series B} {\bf 66} (1996) 87--122.

\bibitem{erdos-graham} Erd\"os, P., Graham, R. L.: Problem proposed at the 6th Hungarian Combinatorial Conference (Eger, July 1981).


\bibitem{gardner} Gardner, R. J.: {\em Geometric Tomography} Cambridge University Press, 1995.

\bibitem{ghorvath} G. Horv\'ath, \'A.: Strongly self-dual polytopes and distance graphs in the unit sphere {\em Acta Mathematica Hungarica} {\bf 163} (2021) 640–651. DOI:10.1007/s10474-020-01106-6

\bibitem{jensen} Jensen, A.: {\it Self-polar polytopes} In Contemporary Mathematics 764 Polytopes and Discrete Geometry 2021.

\bibitem{katz} Katz, M.: Diameter-Extremal Subsets of Spheres \emph{ Discrete Comput Geom} {\bf 4} (1989) 117--137. 

\bibitem{katz-memoli-wang} Katz, M., Mémoli, F., Wang, Q.: Extremal spherical polytopes and Borsuk's conjecture.  \url{https://arxiv.org/abs/2301.13076} 


\bibitem{lovasz} Lov\'asz, L.: Strongly self-dual polytopes and the chromatic number of distance graph on the sphere \emph{Acta Sci. Math.} {\bf 45} (1983) 317--323.

\bibitem{raigorodskii} Raigorodskii, A. M.: Coloring Distance Graphs and Graphs of Diameters \emph{in Thirty Essays on Geometric Graph Theory eds.: J\'anos Pach} Springer New York Heidelberg Dordrecht London, 2013 DOI: 10.1007/978-1-4614-0110-0.


\bibitem{ziegler} Ziegler, G. M.: {\it Lectures on polytopes} Graduate Texts in Mathematics, {\bf 152}, (Springer-Verlag New York, 1995.)

\end{thebibliography}
\end{document}